\documentclass[draft]{amsart}

\usepackage{amsmath, amsthm, amscd, amsfonts, amssymb, graphicx, color}
\usepackage{datetime}\usdate

\newcommand{\ulp}{{\textup{(}}}
\newcommand{\urp}{{\textup{)}}}
\newcommand{\urps}{{\textup{)} }}

\newcommand{\lep}{\left (}
\newcommand{\rip}{\right )}
\newcommand{\lev}{\left\vert}
\newcommand{\riv}{\right\vert}
\newcommand{\leV}{\left\Vert}
\newcommand{\riV}{\right\Vert}
\newcommand{\leb}{\left\{\,}
\newcommand{\rib}{\,\right\}}
\newcommand{\lea}{\left\langle}
\newcommand{\ria}{\right\rangle}
\newcommand{\leh}{\left[}
\newcommand{\rih}{\right]}

\newcommand{\R}{{\mathbb{R}}}

\newcommand{\N}{{\mathbb{N}}}

\newcommand{\rest}[1]{{\restriction_{#1}}}

\newcommand{\supp}{{\textup{supp\,}}}

\newcommand{\rbs}{{\pi}}
\newcommand{\rbt}{{\pi}}
\newcommand{\rct}{{\pi}}
\newcommand{\erbt}{{\rbt_\proj}}
\newcommand{\erbttemp}{{\tilde\rbt}}

\newcommand{\ts}{X}
\newcommand{\tselt}{\xi}
\newcommand{\set}{X}
\newcommand{\setelt}{\xi}

\newcommand{\bl}{E}
\newcommand{\elt}{x}

\newcommand{\centersymbol}{Z}

\newcommand{\bs}{E}

\newcommand{\dual}[1]{#1^\prime}
\newcommand{\bidual}[1]{#1^{\prime\prime}}
\newcommand{\dualelt}{\dual{\elt}}

\newcommand{\dualbs}{\dual{\bs}}
\newcommand{\dualbl}{\dual{\bl}}
\newcommand{\bidualbl}{\bidual{\bl}}
\newcommand{\centerbl}{\centersymbol(\bl)}

\newcommand{\positive}[1]{#1_+}
\newcommand{\negative}[1]{#1_-}

\newcommand{\adjoint}[1]{#1^\prime}

\newcommand{\commutant}[1]{#1^\prime}
\newcommand{\bicommutant}[1]{#1^{\prime\prime}}

\newcommand{\blpos}{\positive{\bl}}

\newcommand{\dualblpos}{\positive{\dualbl}}

\newcommand{\condelt}{\elt\in\bl}
\newcommand{\condeltpos}{\elt\in\blpos}
\newcommand{\conddualelt}{\dualelt\in\dualbl}

\newcommand{\condboth}{\elt\in\bl,\dualelt\in\dualbl}
\newcommand{\condbothsp}{\elt\in\bl,\,\dualelt\in\dualbl}

\newcommand{\condbothpossp}{\elt\in\blpos,\,\dualelt\in\dualblpos}

\newcommand{\standardpair}{\elt,\dualelt}

\newcommand{\pairing}[1]{\lea #1\ria}

\newcommand{\idmap}{\textup{\bf id}}

\newcommand{\onefunction}{\mathbf 1}
\newcommand{\zerofunction}{0}
\newcommand{\zerofunctionset}{\zerofunction}
\newcommand{\onefunctionset}{\onefunction_\set}
\newcommand{\zerofunctionts}{\zerofunction}
\newcommand{\onefunctionts}{\onefunction_\ts}

\newcommand{\linear}{\mathcal L}
\newcommand{\bounded}{\mathcal L}
\newcommand{\regular}{\linear_{\textup{r}}}

\newcommand{\boundedop}[1]{\bounded(#1)}
\newcommand{\regularop}[1]{\regular(#1)}

\newcommand{\boundedbs}{\boundedop{\bs}}

\newcommand{\boundedbl}{\boundedop{\bl}}
\newcommand{\regularbl}{\regularop{\bl}}

\newcommand{\norm}[1]{{\leV #1\riV}}
\newcommand{\regnorm}[1]{{\leV #1\riV}_{\textup r}}
\newcommand{\abs}[1]{\left\vert #1\right\vert}

\newcommand{\normfunction}{\norm{\,.\,}}
\newcommand{\regnormfunction}{\regnorm{\,.\,}}

\newcommand{\boundedblwithnorm}{\lep\boundedbl,\normfunction\rip}
\newcommand{\regularblwithnorm}{\lep\regularbl,\normfunction\rip}
\newcommand{\regularblwithregnorm}{\lep\regularbl,\regnormfunction\rip}

\newcommand{\mset}{\Delta}
\newcommand{\proj}{P}
\newcommand{\maxprojset}{\proj(\set)}
\newcommand{\maxprojts}{\proj(\ts)}

\newcommand{\mupair}[1]{\mu_{#1}}
\newcommand{\mustandardpair}{\mupair{\standardpair}}
\newcommand{\intts}{\int_\ts}

\newcommand{\cont}[1]{C(#1)}
\newcommand{\conto}[1]{C_0(#1)}
\newcommand{\contc}[1]{C_{\textup c}(#1)}

\newcommand{\contts}{\cont{\ts}}
\newcommand{\contots}{\conto{\ts}}
\newcommand{\contcts}{\contc{\ts}}

\newcommand{\sigmaset}{\Omega}
\newcommand{\sigmats}{{\Omega}}

\newcommand{\boundedletter}{\mathcal B}
\newcommand{\simpleletter}{\mathcal S}
\newcommand{\simpleset}{\simpleletter(\set)}
\newcommand{\simplets}{\simpleletter(\ts)}
\newcommand{\boundedset}{\boundedletter(\set)}
\newcommand{\boundedts}{\boundedletter(\ts)}

\newcommand{\boundedregmeas}[1]{M_{\textup b}(#1)}
\newcommand{\boundedregmeasts}{\boundedregmeas{\ts}}

\newcommand{\charfunc}{\chi}

\theoremstyle{plain}

\newtheorem{theorem}{Theorem}[section]
\newtheorem{proposition}[theorem]{Proposition}
\newtheorem{lemma}[theorem]{Lemma}
\newtheorem{corollary}[theorem]{Corollary}

\theoremstyle{definition}

\newtheorem{definition}[theorem]{Definition}

\newtheorem{remark}[theorem]{Remark}

\numberwithin{equation}{section}

\begin{document}


\title [Positive representations of $C_0(X)$]{Positive representations of $\boldsymbol{C_0(X)}$.\! I.}

\author{Marcel de Jeu}
\address{Marcel de Jeu, Mathematical Institute, Leiden University, P.O.\ Box 9512, 2300 RA Leiden, The Netherlands}
\email{mdejeu@math.leidenuniv.nl}

\author{Frejanne Ruoff}
\address{Frejanne Ruoff}
\email{frejanne.ruoff@gmail.com}




\begin{abstract}
We introduce the notion of a positive spectral measure on a $\sigma$-algebra, taking values in the positive projections on a Banach lattice. Such a measure generates a bounded positive representation of the bounded measurable functions.
If $X$ is a locally compact Hausdorff space, and $\pi$ is a positive representation of $C_0(X)$ on a KB-space, then $\pi$ is the restriction to $C_0(X)$ of such a representation generated by a unique regular positive spectral measure on the Borel $\sigma$-algebra of $X$. The relation between a positive representation of $C_0(X)$ on a Banach lattice and\textemdash if it exists\textemdash a generating positive spectral measure on the Borel $\sigma$-algebra is further investigated; here and elsewhere phenomena occur that are specific for the ordered context.
\end{abstract}

\maketitle


\section{Introduction and overview}\label{sec:intro}

Suppose that  $\ts$ is a locally compact Hausdorff space, with $\contots$ denoting the ordered Banach algebra of all real-valued continuous functions vanishing at infinity. Positive representations of such an algebra $\contots$ on a Banach lattice $\bl$, i.e.\ representations such that positive functions act as positive operators on $\bl$, are quite common. Rather trivially, $\contots$ acts positively by multiplication on many Banach lattices of (equivalence classes of) functions on $\ts$. Somewhat in disguise, since the center $\centerbl$ of an arbitrary Banach lattice is such a space for compact $\ts$, there is also a positive representation of this type associated with every Banach lattice. Such positive representations also occur in a very general context where ordering is not present from the start: If $\bs$ is an arbitrary Banach space, $\ts$ is compact, and $\rct$ is a bounded unital representation of $\contots$ on $\bs$, then every cyclic closed subspace of $\bs$ can be supplied with an ordering and an equivalent norm so that it becomes a Banach lattice on which $\contots$ acts positively, cf.\ \cite[Lemma~4.6]{abrarekit} or  \cite[Proposition~2.5]{pagric}; according  to \cite[Proof of Lemma~4.6]{abrarekit} this result goes back to \cite{kai}.

Given the ubiquity of their occurrence, it is natural to ask which positive representations of $\contots$ on a Banach lattice $\bl$ are generated by a (regular) positive spectral measure on the Borel $\sigma$-algebra $\sigmats$ of $\ts$; here `positive' refers to the fact that the  measure takes values in the positive projections on $\bl$. It is one of our main results that this is always the case if $\bl$ is a KB-space, cf.\ Theorem~\ref{thm:generating_measure_KB}.

The study of positive spectral measures and of positive representations of $\contots$ (and also that of positive representations of algebras of bounded measurable functions) appears to be new. The main related results in the literature that we are aware of are concerned with bounded unital representations of $\contots$ for compact $X$ on a Banach space.  Thus the class of representation spaces is more general, but the class of representations is more restrictive:  our $\ts$ need not be compact, and, even if it is, our representation need not be unital. Of course, one can extend a bounded representation $\rct$ of $\contots$ for non-compact $X$ to a bounded unital representation $\rct_\infty$ of $\conto{\ts_\infty}$, where $\ts_\infty$ is the one-point compactification of $\ts$. Thus one places oneself in the convenient unital and compact framework, but removing (and hence also masking) all possible non-degeneracy of the original representation from the very start seemed counterproductive to us. Hence we have considered the most general case, which for positive representations in Banach lattices turns out to be feasible.

The presence of ordered structures brings about several new phenomena as compared to the Banach space context, some of them already at a very basic level. For example, boundedness of positive representations tends to be automatic, and it is often possible to give exact expressions for their norms (cf.\ Propositions~\ref{prop:automatic_continuity_set}, ~  \ref{prop:automatic_continuity_ts}, and \ref{prop:norm_for_generated_representation}). Furthermore, relations of a type that is particular to the ordered context make their appearance. For example, if $\proj$ is a positive spectral measure on $\sigmats$ that is regular in the (usual) sense of Definition~\ref{def:regular_spectral_measure}, and $\mset$ is a Borel set, then
\[
\proj(\mset)=\inf\leb \proj(V): V \textup{ open and }\mset\subset V\rib
\]
in the regular operators on $\bl$, cf.\ Proposition~\ref{prop:regularity_of_the_spectral_measure}. Therefore $\proj$ is outer regular as a map from the $\sigmats$ to the regular operators on $\bl$; likewise, $\proj$ is inner regular in this sense. In a similar vein, if $\pi$ is a positive representation of $\contots$ on $\bl$ that has a generating regular positive spectral measure $\proj$ on $\sigmats$, and $V$ is open, then
\begin{equation*}
\proj(V)=\sup \leb\rct(\phi) : \phi\in\contcts,\,\supp \phi\subset V,\, \zerofunctionts\leq \phi\leq\onefunctionts\rib
\end{equation*}
in the regular operators on $\bl$, cf.\ Theorem~\ref{thm:construct_everything_from_rcts} for this and related results. As in the previous example, this statement, which is reminiscent of (and stems from) a well known formula in the circle of ideas around the Riesz representation theorem, is only meaningful in an ordered context, as it tells us how to determine $\proj$ from $\rct$ in a purely order-theoretical manner.

\medskip

The reader who is familiar with the theory of spectral measures in Banach spaces will without difficulty recognise several ideas and techniques employed in this paper. We have, nevertheless, still included the detailed proofs of our results, as these proofs are inspired by, but not identical, to those in the general case. The ordered context has its own peculiarities from the very start, and it seemed to us that the only way to exploit this extra information to its full potential, while still producing a readable paper, was to make a fresh start and give precise arguments. It has the additional advantage that it makes the paper essentially self-contained. We have given pertinent references to the general theory to the best of our abilities.

\medskip
This paper is organised as follows.

Section~\ref{sec:preliminaries} contains the basis notation and terminology, as well as an overview of material around the Riesz representation theorem. We need more than the mere existence of measures representing bounded linear functionals, and we have included what we need for the ease of reference, and also to establish terminology (which is not uniform throughout the literature).

Section~\ref{sec:positive_spectral_measures} introduces positive spectral measures on arbitrary $\sigma$-algebras, and contains some first basic results.

In Section~\ref{sec:from_positive_spectral_measure_to_positive_representation} a bounded positive representation of the bounded measurable functions on a measurable space is constructed from a positive spectral measure. Any positive representation of this algebra is, in fact, always bounded, but when is there an underlying positive spectral measure?
If $\bl$ has $\sigma$-order continuous norm, then we can characterise the positive representations thus obtained: they are precisely the $\sigma$-order continuous ones.  It is not true that an arbitrary positive representation of the bounded measurable functions has a generating positive spectral measure on the pertinent $\sigma$-algebra; this will be taken up in \cite{dejjia}.

Section~\ref{sec:existence_of_generating_regular_positive_spectral_measure} focusses on the topological context of a locally compact Hausdorff space $\ts$. We are concerned with automatic continuity of positive representations of $\contots$, and show that every positive representation of $\contots$ on a KB-space has a generating regular positive spectral measure on $\sigmats$, cf.\ Theorem~\ref{thm:generating_measure_KB}. It is indicated in Remark~\ref{rem:alternative_approaches} how this result can also be derived using Banach space results and Banach space properties of KB-spaces. This alternative approach is more involved and lacks the order-theoretical flavour of ours. Additional regularity properties of a regular positive spectral measures are investigated, as well as the relation between a positive representation of $\contots$ and\textemdash if there is one\textemdash a generating possibly regular positive spectral measure on $\sigmats$ for that representation.

\medskip

This paper, which we consider to be part of the groundwork for the theory of positive representations of $\contots$, is expected to have a sequel \cite{dejjia}. There are still several issues to be investigated. The relation between the degeneracy of the representation of $\pi$ and the projection corresponding to the element $\ts$ of $\sigmats$ is one of these; Remarks~\ref{rem:preview_1} and \ref{rem:preview_2} contain additional ones. Furthermore, it is interesting to turn the tables: if $\bl$ is a Banach lattice such that every positive representation of every $\contots$-space (where $\ts$ is a locally compact Hausdorff space) on $\bl$ has a generating regular positive spectral measure on the pertinent Borel $\sigma$-algebra, then what can one say about $\bl$? Must $\bl$ be a KB-space, i.e.\ is Theorem~\ref{thm:generating_measure_KB} optimal?

For an at this moment more distant perspective, we recall that the spectral theorem for representations of commutative $C^*$-algebras on Hilbert spaces underlies, in the end, a structure theorem (even a classification theorem) for an arbitrary normal operator on a separable Hilbert space, cf.\ \cite[Theorem~IX.10.1]{con}. Although still a considerable effort would be needed, and success is not guaranteed (presumably one would need a reasonable analogue of von Neumann algebras), it could be hoped that, analogously, Theorem~\ref{thm:generating_measure_KB} could be the basis for a structure theorem for an arbitrary orthomorphism on suitable KB-spaces.

\section{Preliminaries}\label{sec:preliminaries}

We start by collecting some basic notation and facts, as well as giving a precise overview of the Riesz representation theorem and the facts surrounding it.

\subsection{Basics}\label{subsec:basics}
All vector spaces in this paper are over the real numbers. If $V$ is an ordered vector space, we write $\positive{V}$ and $\negative{V}$ for its positive and negative cone, respectively.

If $\bs$ is a Banach space, then $\dualbs$ denotes its norm dual. We write $\boundedbs$ for the bounded linear operators on $\bs$, and $\idmap_\bs$ for the identity map on $\bs$. If $\bl$ is a Banach lattice, then $\regularbl$ will denote the vector space of regular operators on $\bl$, i.e.\ the linear operators on $\bl$ that can be written as a difference of two positive linear operators on $\bl$. It is well known that every positive linear operator on $\bl$ is bounded, so that $\regularbl\subset\boundedbl$, and that, for $T\geq 0$, $\norm{T}=\sup\leb\norm{Tx} : \elt\in\blpos,\norm{x}\leq 1\rib$, where $\blpos$ is the positive cone of $\bl$ \cite[Proposition~1.3.5]{mey}. Consequently, the operator norm is monotone on the positive cone $\positive{\regularbl}$ of $\regularbl$: if $T_1,T_2\in\regularbl$ and $0\leq T_1\leq T_2$, then $\norm{T_1}\leq\norm{T_2}$.

If $\bl$ is a Dedekind complete Banach lattice, then $\regularbl$ is a Dedekind complete Banach lattice when supplied with the natural ordering and the regular norm $\regnormfunction$, defined by $\regnorm{T}=\norm{\abs{T}}$ ($T\in\regularbl$), cf.\ \cite[Theorem~4.74]{aliburpos}. If $\bl\neq\leb 0\rib$ the inclusion map $\regularblwithregnorm\hookrightarrow\boundedblwithnorm$ has norm 1 \cite[p.~255]{aliburpos}.

If $F$ is a normed space, and $\bl$ is a Banach lattice, then the norm of a bounded linear map $\rbs \colon F\to\boundedblwithnorm$ or $\rbs \colon F\to\regularblwithnorm$ will be denoted by $\norm{\rbs}$. If $\bl$\ is Dedekind complete, then the norm of a bounded linear map $\rbs \colon F\to\regularblwithregnorm$ will be denoted by $\regnorm{\rbs}$.

\begin{definition}
If $A$ is an ordered algebra, and $\bl$ is a Banach lattice, then a \emph{positive representation of $A$ on $\bl$} is an algebra homomorphism $\rbs \colon A \to\boundedbl$ such that $\rbs(\positive{A})\subset\positive{\regularbl}$. If $A$ is normed, we do not require $\rbs$ to be bounded. If $A$ is unital, we do not require $\rbs$ to be unital. The linear span of $\{\,\rbs(a)x : a\in A,\,x\in\bl \,\}$ need not be dense in $\bl$, i.e.\ the representation can be degenerate.
\end{definition}

\begin{remark}\label{rem:preview_1}
In the cases of our interest, the ordered algebra $A$ is in fact a lattice, so that a positive representation of $A$ on $\bl$ maps $A$ into $\regularbl$. If $\bl$ is Dedekind complete, then $\regularbl$ is also a lattice, and it is meaningful to require that $\pi \colon A \to\regularbl$ is a lattice homomorphism. We emphasize that this is not required in the present paper, but in \cite{dejjia} we shall investigate to which extent this is automatically the case.
\end{remark}

If $V$ is a not necessarily order complete ordered vector space, and $(v_n)_{n=1}^\infty\subset V$, then we shall write $v_n\uparrow v$ if $v_n\leq v_{n+1}$ ($n\geq 1$) and $v=\sup_n v_n$.

If $\set$ is a set, then $\onefunctionset$ will denote the constant function 1 on $\set$. If $\sigmaset$ is a $\sigma$-algebra of subsets of $\set$, then $\boundedset$ will denote the bounded $\sigmaset$-measurable functions on $\set$. It is a Banach lattice algebra, and we let $\simpleset$ denote the lattice subalgebra of simple functions. The order bounded subsets of $\boundedset$ are precisely the norm bounded ones. If $\lep\phi_n\rip_{n=1}^\infty\subset\boundedset$ and $\phi\in\boundedset$, then $\phi_n\uparrow\phi$ in $\boundedset$ if and only if $\phi_n(\setelt)\uparrow\phi(\setelt)$ for all $\setelt\in\set$.

The supremum norm of $\phi\in\boundedset$ is written as $\norm{\phi}$.

\subsection{Regular Borel measures}
If $\ts$ is a locally compact Hausdorff space, we let $\sigmats$ denote the Borel $\sigma$-algebra generated by the open sets, and $\boundedts$ the Banach lattice algebra of bounded Borel measurable functions on $\ts$. We write $\contcts$ for the normed lattice algebra of continuous functions on $\ts$ with compact support, and $\contots$ for the Banach lattice algebra of continuous functions vanishing at infinity. If $\ts$ is compact, we shall still write $\contots$ rather than $\contts$ for the sake of uniform terminology and notation.

The results in Section~\ref{sec:existence_of_generating_regular_positive_spectral_measure} rely heavily on the Riesz representation theorem and the general theory of regular Borel measures on locally compact Hausdorff spaces. Since the terminology in this field is not entirely standardised, and we need a bit more than the bare minimum, we give precise definitions and an overview of what we need.

Adapting the terminology from \cite[p.~352]{aliburpri}, we say that a measure $\mu \colon \sigmats \to [0,\infty]$ is:
 \begin{enumerate}
 \item a \emph{Borel measure} if  $\mu(K)<\infty$ for all compact $K\subset\ts$;
 \item \emph{outer regular on $\mset\in\sigmats$} if $\mu(\mset)=\inf\leb \mu(V): V \textup{ open and }\mset\subset V\rib$;
 \item \emph{inner regular on $\mset\in\sigmats$} if $\mu(\mset)=\sup\leb \mu(K): K \textup{ compact and }K\subset\mset\rib$;
 \item \emph{a regular Borel measure} if it is  a Borel measure that is outer regular on all $\mset\in\sigmats$ and inner regular on all open subsets of $\ts$.
 \end{enumerate}
  The nomenclature is not uniform in the literature; sometimes the inner regularity on all elements of $\sigmats$ rather than on just the compact subsets is incorporated in the definition of a regular Borel measure, as in \cite[p.~212]{fol}. In \cite[p.~212]{fol} our regular Borel measures are called Radon measures.

  We let $\boundedregmeasts$ be the regular finite signed Borel measures on $\sigmats$, i.e.\ the finite signed measures on $\sigmats$ that can be written as a difference of two regular finite Borel measures. Then $\boundedregmeasts$ is a Banach lattice when supplied with the natural ordering and the total variation norm. If $\mu\in\boundedregmeasts$, then the linear functional $I_\mu \colon \contots \to\R$, defined by $I_\mu(\phi)=\intts\phi\,d\mu$ ($\phi\in\contots$), is bounded.

Most of the results we need are collected in the following overview theorem. Part (1) is the combination of \cite[Theorem~38.7]{aliburpri} and \cite[Theorem~7.17]{fol}; part (2) follows from \cite[Corollary~7.6]{fol}; part (3) is contained in \cite[Theorem~7.2]{fol}.

\begin{theorem}\label{thm:overview_theorem}
 Let $\ts$ be a locally compact Hausdorff space with Borel $\sigma$-algebra $\sigmats$. Then:
 \begin{enumerate}
  \item \ulp Riesz representation theorem\urps The map $\mu\mapsto I_\mu$ is an isometric order isomorphism between the Banach lattices $\boundedregmeasts$ and $\dual{\contots}$.
  \item If $\mu\in\positive{\boundedregmeasts}$, then $\mu$ is inner regular on all elements of $\sigmats$.
  \item  If $\mu \colon \sigmats \to [0,\infty]$ is a regular Borel measure on $\sigmats$, and $V$ is an open subset of $\ts$, then
  \begin{equation}\label{eq:open_versus_functions}
  \mu(V)=\sup\leb\intts\phi\,d\mu:\phi\in\contcts,\,\supp\phi\subset V,\, \zerofunctionts\leq\phi\leq\onefunctionts\rib,
  \end{equation}
  and if $K$ is a compact subset of $\ts$, then
  \begin{equation}\label{eq:compact_versus_functions}
\mu(K)=\inf \leb\intts\phi\,d\mu : \phi\in\contcts,\,\phi\geq\charfunc_K\rib.
\end{equation}
\end{enumerate}
\end{theorem}

In \eqref{eq:compact_versus_functions}, and analogously elsewhere, $\charfunc_K$ denotes the characteristic function of the subet $K$ of $\set$.

The following fact is implied by \cite[Problem~38.12]{aliburpro}.

\begin{lemma}\label{lem:again_regular_measure}
Let $\ts$ be a locally compact Hausdorff space with Borel $\sigma$-algebra $\sigmats$. Let $\mu\in\boundedregmeasts$ and $\psi\in\boundedts$, and define $\mu^ \psi \colon \sigmats\to\R$ by $\mu^\psi(\mset)=\int_\mset \psi\,d\mu$ \ulp$\mset\in\sigmats$\urp. Then $\mu^\psi\in\boundedregmeasts$.
\end{lemma}

The following consequence of Lemma~\ref{lem:again_regular_measure} will be used in the proof of Theorem~\ref{thm:generating_measure_KB}.

\begin{lemma}\label{lem:then_for_bounded_borel}
Let $\ts$ be a locally compact Hausdorff space with Borel $\sigma$-algebra $\sigmats$, and let $\psi\in\boundedts$. Suppose $\mu,\nu\in \boundedregmeasts$ are such that
\begin{equation}\label{eq:equality_of_measures}
\intts \phi\psi\,d\mu=\int\phi\, d\nu
\end{equation}
for all $\phi\in\contots$. Then \eqref{eq:equality_of_measures} holds for all $\phi\in\boundedts$.
\end{lemma}

\begin{proof}
If we let $\mu^\psi(\mset)=\int_\mset \psi\,d\mu$ ($\mset\in\sigmats$), then \eqref{eq:equality_of_measures} implies that $\intts\phi\,d\mu^\psi=\intts\phi\psi\,d\mu=\intts\phi\,d\nu$ for all $\phi\in\contots$. Since Lemma~\ref{lem:again_regular_measure} asserts that $\mu^\psi$ is again a regular finite signed Borel measure, the uniqueness statement in the Riesz representation theorem shows that $\mu^\psi=\nu$. But then $\intts\phi\,d\mu^{\psi}=\intts \phi\,d\nu$ for all $\phi\in\boundedts$, i.e.\ \eqref{eq:equality_of_measures} holds for all $\phi\in\boundedts$.
\end{proof}

\subsection{Consequence of a monotone class theorem}\label{subsec:monotone_class_theorem}

In the proof of Theorem~\ref{thm:generating_measure_KB} we shall need the following special case of a monotone class theorem from measure theory \cite[Theorem~3.14]{wil}.

\begin{theorem}\label{thm:monotone_class_theorem}
        Let $\ts$ be a locally compact Hausdorff space with Borel $\sigma$-algebra $\sigmats$. Suppose $L$ is a vector space of bounded functions on $\ts$ such that:
        \begin{enumerate}
                \item $\charfunc_V\in L$ for all open subsets $V$ of $\ts$;
                \item If $\phi_n$ is a sequence of functions in $L$, and if $\phi$ is a bounded function on $\ts$ such that $0\leq \phi_n(\tselt)\uparrow \phi(\tselt)$ for all $\tselt\in\ts$, then $\phi\in L$.
        \end{enumerate}
        Then $\boundedts\subset L$.
\end{theorem}
\section{Positive spectral measures}\label{sec:positive_spectral_measures}

In this section we introduce the notion of a positive spectral measure on a general $\sigma$-algebra, and establish some basic properties.

\begin{definition}\label{def:positive_spectral_measure}
Let $\set$ be a set, $\sigmaset$ a $\sigma$-algebra of subsets of $\set$, and $\bl$ a Banach lattice. A map $\proj \colon \sigmaset \to\regularbl\subset\boundedbl$ is called a \emph{positive spectral measure on $\sigmaset$} when it has the following properties:
\begin{enumerate}
 \item For each $\mset$ in $\sigmaset$, $\proj(\mset)$ is a positive projection,
 \item $\proj(\emptyset)=0$,
 \item $\proj(\mset_1\cap\mset_2)=\proj(\mset_1)\proj(\mset_2)$ for $\mset_1, \mset_2\in\sigmaset$,
 \item $\proj$ is $\sigma$-additive for the strong operator topology on $\boundedbl$, i.e.\ if $(\mset_n)_{n=1}^{\infty}$ are pairwise disjoint elements of $\sigmaset$, and $\elt\in\bs$, then
 \[
\proj(\cup_{n=1}^{\infty} \mset_n)\elt=\sum_{n=1}^{\infty} \proj(\mset_n) \elt,
 \]
where the series converges in the norm topology.
\end{enumerate}
If $\proj(\set)=\idmap_{\bs}$, then $\proj$ is called \emph{unital}.
\end{definition}

\begin{remark}\label{rem:Pettis_consequence}
It is instrumental for the proof of Theorems~\ref{thm:order_continuity_is_necessary_and_sufficient} and~\ref{thm:generating_measure_KB} that, as an immediate consequence of a theorem of Pettis' \cite[Theorem~IV.10.1] {dunsch1} (see also \cite[Lemma III.2]{ric}), the combination of (1), (2), (3), and (4) is equivalent to the combination of (1), (2), (3), and ($4^\prime$), where
\begin{enumerate}
 \item[($4^\prime$)] $\proj$ is $\sigma$-additive for the weak operator topology on $\boundedbl$, i.e.\ if $(\mset_n)_{n=1}^{\infty}$ are pairwise disjoint elements of $\sigmaset$, and $\elt\in\bs,\dualelt\in\dualbl$, then
 \[
 \pairing{\proj(\cup_{n=1}^{\infty} \mset_n)\elt,\dualelt}=\sum_{n=1}^{\infty} \pairing{\proj(\mset_n)\elt,\dualelt}.
 \]
\end{enumerate}
\end{remark}

\begin{remark}\quad
\begin{enumerate}
\item There are several variants of the definition of spectral measures (or resolutions of the identity) to be found in the literature:
\begin{enumerate}
\item In \cite[Definition~IX.1.1]{con}, \cite[Definition~12.17]{rud} (both in the context of Hilbert spaces), and \cite[Definition~III.4]{ric} (in the context of  Banach spaces) a spectral measure is required to be $\sigma$-additive in the strong operator topology (or, equivalently, in the weak operator topology), as we do. Contrary to our definition, however, it is required to be unital.
\item In the context of Banach spaces, a spectral measure is defined on a Boolean algebra of subsets of a set in \cite[Definition~XV.2.1]{dunsch3}. A spectral measure in that sense on a $\sigma$-algebra is required to be only finitely additive, whereas we require $\sigma$-additivity. Contrary to our definition, it is required to be unital.
\end{enumerate}
\item Using the terminology of \cite[p.~1]{dieuhl}, a positive spectral measure on $\sigmaset$ as in our definition is a map $\proj\colon\sigmaset\to\regularbl$ such that it takes values in the positive projections on $\bl$, $\proj(\mset_1\cap\mset_2)=\proj(\mset_1)\proj(\mset_2)$ ($\mset_1,\,\mset_2\in\sigmaset$), and such that, for all $\elt\in\bl$, the map $\mset\mapsto\proj(\mset)x$ ($\mset\in\sigmaset$) is a countably additive vector measure on $\sigmaset$ with values in $\bl$.
\end{enumerate}
\end{remark}

It follows easily from Definition~\ref{def:positive_spectral_measure} that a positive spectral measure is finitely additive, and the following lemma on monotonicity and uniform boundedness is then clear.

\begin{lemma}\label{lem:measure_is_bounded}
Let $\proj \colon \sigmaset \to\regularbl$ be a positive spectral measure. Then $\proj$ is monotone, i.e.\ if $\mset_1,\mset_2\in\sigmaset$ and $\mset_1\subset\mset_2$, then $0\leq \proj(\mset_1)\leq\proj(\mset_2)$. Consequently, $\norm{\proj(\mset_1)}\leq \norm{\proj(\mset_2)}$ for such $\mset_1,\mset_2$, and in particular $\norm{\proj(\mset)}\leq\norm{\maxprojset}$ for all $\mset\in\sigmaset$.
\end{lemma}

\begin{remark}\quad
\begin{enumerate}
\item It is a consequence of the uniform boundedness principle that a ($\sigma$-additive) spectral measure taking values in the bounded projections on a Banach space is always uniformly bounded (cf.\ \cite[Lemma~III.3]{ric}; the proof also applies if the spectral measure is not unital). The point in Lemma~\ref{lem:measure_is_bounded} is, therefore, the explicit and sharp uniform upper bound $\norm{\maxprojset}$.
\item
If $\proj$ is unital, then $\proj(\sigmaset)$ is a $\sigma$-complete Boolean algebra of bounded projections on the Banach space $\bl$. In that case, \cite[Lemma~XVII.3.3]{dunsch3} also implies that  $\proj(\sigmaset)$ is uniformly bounded.
\end{enumerate}
\end{remark}

If $\proj \colon \sigmaset \to\regularbl$ is a positive spectral measure, and $\condbothsp$, we define $\mustandardpair \colon \sigmaset\to\R$ by $\mustandardpair(\mset)=\pairing{\proj(\mset)\elt,\dualelt}$ ($\condbothsp$); our notation $\mustandardpair$ follows Schaefer, cf.\ \cite[proof of Proposition~V.3.2]{schbook}. It is clear that $\mustandardpair$ is a finite signed measure, and that $\mustandardpair$ is positive when $\condbothpossp$. By a standard argument \cite[p.~97]{dunsch1}, we have $\norm{\mustandardpair}\leq 2\sup_{\mset\in\sigmaset}\abs{\mustandardpair(\mset)}$, so that $\norm{\mustandardpair}\leq 2\norm{\maxprojset}\norm{\elt}\norm{\dualelt}$ ($\condboth$). The factor 2 can be removed here, as is stated in the following still more precise result.

\begin{lemma}\label{lem:total_variation_estimate}
Let $\proj \colon \sigmaset \to\regularbl$ be a positive spectral measure, and $\condbothsp$. Then $\norm{\mustandardpair}\leq \pairing{\maxprojset\abs{\elt},\abs{\dualelt}}$. Equality holds if $\elt\in\blpos\cup\negative{\bl}$ and $\dualelt\in\dualblpos\cup\negative{\dualbl}$.

For all $\condbothsp$, $\norm{\mustandardpair}\leq \norm{\maxprojset}\norm{\elt}\norm{\dualelt}$.
\end{lemma}

\begin{proof}
Let $\condbothpossp$. If $\set=\bigcup_{i=1}^n\mset_i$ is a measurable disjoint partition of $\set$, then, using the fact that the $\proj(\mset_i)$ are positive, we see that
\begin{align*}
\sum_{i=1}^n\abs{\mustandardpair(\mset_i)}&=\sum_{i=1}^n\abs{\pairing{\proj(\mset_i)\elt,\dualelt}}\\
&=\sum_{i=1}^n \pairing{\proj(\mset_i)\elt,\dualelt}\\
&=\pairing{\sum_{i=1}^n\proj(\mset_i)\elt,\dualelt}\\
&=\pairing{\maxprojset \elt,\dualelt}.
\end{align*}
Hence $\norm{\mustandardpair}=\pairing{\maxprojset\elt,\dualelt}$ for $\elt\in\blpos$ and $\dualelt\in\dualblpos$.  This implies that $\norm{\mustandardpair}= \pairing{\maxprojset\abs{\elt},\abs{\dualelt}}$ if $\elt\in\blpos\cup\negative{\bl}$ and $\dualelt\in\dualblpos\cup\negative{\dualbl}$. The rest of the lemma follows by splitting arbitrary $\elt\in\bl$ and $\dualelt\in\dualbl$ into their positive and negative parts.
\end{proof}

The following fact will be needed in the proof of Proposition~\ref{prop:regularity_of_the_spectral_measure}, where it is shown that regular positive spectral measures on the Borel $\sigma$-algebra of a locally compact Hausdorff space are also regular in a natural sense that is specific for the ordered context.

\begin{lemma}\label{lem:inf_and_sup_msets}
Let $\sigmaset$ be a $\sigma$-algebra of subsets of a set $\set$, $\bl$ a Banach lattice, and $\proj \colon \sigmaset\to\regularbl$ a positive spectral measure with associated finite signed measures $\mustandardpair$ \ulp$\condbothsp$\urps on $\sigmaset$.
\begin{enumerate}
\item Suppose $\mset,\mset_i\in I$ \ulp $i\in I$\urps are elements of $\sigmaset$ such that:
\begin{enumerate}
\item $\mset\subset\mset_i$ \ulp $i\in I$\urp;
\item $\mustandardpair(\mset)=\inf_{i\in I}\mustandardpair(\mset_i)$ for all $\condbothpossp$.
\end{enumerate}
Then $\proj(\mset)=\inf_{i\in I}\proj(\mset_i)$ in $\regularbl$.
\item Suppose $\mset,\mset_i\in I$ \ulp $i\in I$\urps are elements of $\sigmaset$ such that:
\begin{enumerate}
\item $\mset\supset\mset_i$ \ulp $i\in I$\urp;
\item $\mustandardpair(\mset)=\sup_{i\in I}\mustandardpair(\mset_i)$ for all $\condbothpossp$.
\end{enumerate}
Then $\proj(\mset)=\sup_{i\in I}\proj(\mset_i)$ in $\regularbl$.
\end{enumerate}
\end{lemma}

Note that we not assume that $\bl$ is Dedekind complete, hence the existence of the infimum and supremum is not automatic.

\begin{proof}
We prove part (1); the proof of part (2) is similar. Since $\proj$ is monotone, $\proj(\mset)\leq\proj(\mset_i)$ ($i\in I$). Suppose $T\in\regularbl$ and $T\leq\proj(\mset_i)$ ($i\in I$). Let $\condbothpossp$. Then $\pairing{T\elt,\dualelt}\leq\pairing{\proj(\mset_i)\elt,\dualelt}=\mustandardpair(\mset_i)$ ($i\in I$), so that $\pairing{T\elt,\dualelt}\leq\inf_{i\in I} \mustandardpair(\mset_i)=\mustandardpair(\mset)=\pairing{\proj(\mset)\elt,\dualelt}$. Hence $T\leq\proj(\mset)$.
\end{proof}

\section{Positive $\boundedset$-representations generated by positive spectral measures}\label{sec:from_positive_spectral_measure_to_positive_representation}

 Since we know from Lemma~\ref{lem:measure_is_bounded} that a positive spectral measure is uniformly bounded, one can employ a standard method \cite[p.~891-892]{dunsch2}\,\cite[p.~13-14]{ric} to construct a representation of the bounded measurable functions $\boundedset$ on $\set$ from this measure. In the first part of this section, we study the basic properties of the representations thus obtained. In the second part, we turn the tables and ask ourselves which positive representations of $\boundedset$ arise in this fashion.

 Starting with a positive spectral measure $\proj \colon \sigmaset\to\regularbl$ on a $\sigma$-algebra $\sigmaset$ of subsets of a set $\set$, the associated representation $\rbs_\proj \colon \boundedset\to\regularbl$ is constructed as follows. If $\phi=\sum_{i=1}^n \alpha_i\charfunc_{\mset_i}\in\boundedset$ is a simple function, with the $\mset_i$ not necessarily disjoint, then let $\rbs_\proj(\phi)=\sum_{i=1}^n\alpha_i\proj(\mset_i)\in\regularbl$. This is well defined, and one thus obtains a representation $\rbs_\proj$ of the simple functions $\simpleset$ on $\bl$ that is clearly a positive representation of the ordered algebra $\simpleset$. Taking the $\mset_i$ in $\phi=\sum_{i=1}^n \alpha_i\charfunc_{\mset_i}$ to be a measurable disjoint partition of $\set$, and invoking Lemma~\ref{lem:total_variation_estimate} in the penultimate step, we see that
\begin{align*}
\norm{\rbs_\proj(\phi)}&=\sup\leb \abs{\pairing{\sum_{i=1}^n\alpha_i\proj(\mset_i)\elt,\dualelt}}:\elt\in\bl,\,\norm{\elt}\leq 1,\,\dualelt\in\dualbl,\, \norm{\dualelt}\leq 1\rib \\
&\leq\sup\leb\sum_{i=1}^n \abs{\alpha_i}\abs{\pairing{\proj(\mset_i)\elt,\dualelt}}:\elt\in\bl,\,\norm{\elt}\leq 1,\,\dualelt\in\dualbl,\,\norm{\dualelt}\leq 1\rib \\
&\leq\norm{\phi}\sup\leb\sum_{i=1}^n \abs{\pairing{\proj(\mset_i)\elt,\dualelt}}:\elt\in\bl,\,\norm{\elt}\leq 1,\,\dualelt\in\dualbl,\,\norm{\dualelt}\leq 1\rib \\
&\leq\norm{\phi}\sup\leb\norm{\mustandardpair}: \elt\in\bl,\,\norm{\elt}\leq 1,\,\dualelt\in\dualbl,\,\norm{\dualelt}\leq 1\rib\\
&\leq\norm{\phi}\norm{\maxprojset}\\
&=\norm{\phi} \norm{\rbs(\onefunctionset)}.
\end{align*}

We conclude that $\rbs_\proj \colon \simpleset \to\regularblwithnorm\subset\boundedblwithnorm$ is bounded with norm $\norm{\maxprojset}$. Since $\simpleset$ is dense in $\boundedset$, $\rbs_\proj$ extends uniquely to a bounded representation $\rbs_\proj \colon \boundedset \to\boundedblwithnorm$. Furthermore, since $\positive{\simpleset}$ is dense in $\positive{\boundedset}$, and $\positive{\bl}$ is closed in $\bl$, we see that actually $\rbs_\proj(\positive{\boundedset})\subset\positive{\regularbl}$. Hence  $\rbs_\proj \colon \boundedset \to\regularblwithnorm$ is a positive representation with norm $\norm{\maxprojset}$.

If $\bl$ is Dedekind complete, then the regular norm is defined on $\regularbl$. Since $\rbs_\proj$ is positive, $\normfunction$ is monotone on $\positive{\regularbl}$, and $\abs{\phi}\leq\norm{\phi}\onefunctionset$ for $\phi\in\boundedset$, we have $\regnorm{\rbs_\proj(\phi)}=\norm{\abs{\rbs_\proj(\phi)}}\leq
\norm{\rbs_\proj(\abs{\phi})}\leq\norm{\rbs_\proj(\norm{\phi}\onefunctionset)}
=\norm{\maxprojset}\norm{\phi}$. We see that $\rbs_{\proj} \colon \boundedset \to\regularblwithregnorm$ is also bounded, and that $\regnorm{\rbs_\proj}=\norm{\maxprojset}$.

We collect our first results and a few more or less standard properties of $\rbs_\proj$ in the following theorem. At the moment of writing, it is an open question whether the common subalgebra in part (7)(b) can be a proper subalgebra of the common subalgebra in part (5).

\begin{theorem}\label{thm:generated_positive_representation}
Let $\sigmaset$ be a $\sigma$-algebra of subsets of a set $\set$, $\bl$ a Banach lattice, and $\proj \colon \sigmaset\to\regularbl$ a positive spectral measure.
\begin{enumerate}
 \item The map $\rbs_\proj \colon \simpleset \to\regularblwithnorm$, defined on the simple functions by $\rbs_\proj(\sum_{i=1}^n\alpha_j\charfunc_{\mset_i})=\sum_{i=1}^n\alpha_j\proj(\mset_i)$, extends uniquely to a bounded linear map $\rbs_\proj \colon \boundedset \to\boundedblwithnorm$. This extension is a positive representation $\rbs_\proj \colon \boundedset \to\regularblwithnorm$ with norm $\norm{\rbs_\proj}=\norm{\maxprojset}=\norm{\rbs_\proj(\onefunctionset)}$.
 \item $\rbs_\proj$ is unital if and only if $\proj$ is unital.
 \item For $\phi\in\boundedset$, $\rbs_\proj(\phi)\in\regularbl$ is the unique element of $\boundedbl$ such that, for all $\elt\in\bl,\dualelt\in\dualbl$,
\begin{equation}\label{eq:weak_def}
\pairing{\rbs_\proj(\phi)\elt,\dualelt}=\intts \phi\,d\mustandardpair.
\end{equation}
\item If $\phi\in\boundedset$ and $\lep\phi\rip_{n=1}^\infty\subset\boundedset$ is a bounded sequence in $\boundedset$, such that $\lim_{n\to\infty}\phi_n(\setelt)=\phi(\setelt)$ for all $\setelt\in\set$, then $\rbs_\proj(\phi)=\textup{WOT-}\lim_n\rbs_\proj(\phi_n)$.
 \item The closed subalgebras of $\boundedblwithnorm$ generated by $\proj(\sigmaset)$, $\rbs_\proj(\simpleset)$, and $\rbs_\proj(\boundedset)$ are equal.
 \item The commutants $\commutant{\proj(\sigmaset)}$, $\commutant{\rbs_\proj(\simpleset)}$, and $\commutant{\rbs_\proj(\boundedset)}$ in $\boundedbl$ are equal.
 \item If $\bl$ is Dedekind complete, then:
 \begin{enumerate}
 \item $\rbs_\proj \colon \boundedset \to\regularblwithregnorm$ is bounded, with norm $\regnorm{\rbs_\proj}=\norm{\maxprojset}=\norm{\rbs_\proj(\onefunctionset)}$.
 \item The closed subalgebras of $\regularblwithregnorm$ generated by $\proj(\sigmaset)$, $\rbs_\proj(\simpleset)$, and $\rbs_\proj(\boundedset)$ are equal, and this common subalgebra is contained in the common subalgebra in part \ulp 5\urp.
 \end{enumerate}
\end{enumerate}
\end{theorem}

\begin{proof}
For part (3), we note that \eqref{eq:weak_def} holds by construction if $\phi\in\simpleset$. Since, for fixed $\condbothsp$, both sides in \eqref{eq:weak_def} are bounded linear functionals on $\boundedset$, the general case follows by continuity. The uniqueness statement in part (3) is clear, and part (4) is immediate from an application of the dominated convergence theorem. The remaining statements follow easily from the discussion preceding the theorem.
\end{proof}

\begin{remark}\quad
\begin{enumerate}
\item The standard estimate (for the real case) \cite[p.~892]{dunsch2} yields that $\norm{\rbs_\proj}\leq 2\sup_{\mset\in\sigmaset}\norm{\proj(\mset)}$, and combination with Lemma~\ref{lem:measure_is_bounded} then implies that we know a priori that $\norm{\rbs_\proj}\leq 2\norm{\maxprojset}$. Part (1) therefore shows that in the ordered context  one can remove the factor 2 and obtain equality. 
\item If $\bl$ is Dedekind complete, the positive map $\rbs \colon \boundedset\to\regularblwithregnorm$ is automatically bounded (and then so is $\rbs \colon \boundedset\to\regularblwithnorm$. The point in part (7)(a) is the value of $\regnorm{\rbs_\proj}$. 
\item If $\bl$ is Dedekind complete, then there is a seemingly alternative way of obtaining a positive representation of $\boundedset$ on $\bl$. Indeed, one can also view $\rbs_\proj$ on $\simpleset$ as a bounded map $\rbs_\proj \colon \simpleset\to\regularblwithregnorm$, where the codomain is now likewise a Banach space. Extending by continuity, we obtain a bounded positive representation $\rbs_\proj^\textup{r} \colon  \simpleset \to\regularblwithregnorm$. Since the inclusion map $\regularblwithregnorm\hookrightarrow\boundedblwithnorm$ is bounded, a moment's thought shows that actually $\rbs_\proj=\rbs_\proj^{\textup r}$. Hence there is no ambiguity.
\end{enumerate}
\end{remark}

The first statement in the next result is specific for the ordered context. We do not assume that $\bl$ is Dedekind complete, hence the existence of $\sup_{n}\rbs_\proj(\phi_n)$ is not automatic. By \cite[Lemma~1.24]{abrali}, it implies that $\rbs_\proj$ is a $\sigma$-order continuous map between the ordered vector spaces $\boundedset$ and $\regularbl$.

\begin{proposition}\label{prop:order_continuity_is_necessary}
Let $\sigmaset$ be a $\sigma$-algebra of subsets of a set $\set$, $\bl$ a
Banach lattice, and $\proj \colon \sigmaset\to\regularbl$ a positive spectral measure.
If $\phi\in\boundedset$ and $\lep\phi\rip_{n=1}^\infty\subset\boundedset$
is a bounded sequence in $\boundedset$ such that $\phi_n\uparrow\phi$ in
$\boundedset$, then $\rbs_\proj(\phi_n)\uparrow\rbs_\proj(\phi)$ in $\regularbl$, and
$\rbs_\proj(\phi)=\textup{WOT-}\lim_n\rbs_\proj(\phi_n)$.
\end{proposition}

\begin{proof}
Clearly $\rbs_\proj(\phi)\geq\rbs_\proj(\phi_n)$ for all $n$. Suppose $T\in\regularbl$
and $T\geq\rbs_\proj(\phi_n)$ for all $n$. Then, for $\condbothpossp$, we
have $\pairing{T\elt,\dualelt}\geq\pairing{\rbs_\proj(\phi_n)\elt,
\dualelt}=\intts\phi_n\,d\mustandardpair$ for all $n$. The dominated convergence
theorem yields $\pairing{T\elt,\dualelt}\geq\intts\phi\,d\mustandardpair=\pairing{\rbs_\proj(\phi)\elt,\dualelt}$.
Hence $T\geq\rbs_\proj(\phi)$. We have shown that $\rbs(\phi_n)\uparrow\rbs(\phi)$;
the second statement follows from part (4) of Theorem~\ref{thm:generated_positive_representation}.
\end{proof}

Since $\proj(\mset)=\rbs_\proj(\charfunc_{\mset})$ ($\mset\in\sigmaset$), the map $P\mapsto\rbs_\proj$ is injective. This validates the choice of the definite article in the following definition.

\begin{definition}
Let $\sigmaset$ be a $\sigma$-algebra of subsets of a set $\set$, $\bl$ a Banach lattice, $\proj \colon \sigmaset\to\regularbl$ a positive spectral measure, and $\rbs_\proj \colon \boundedset \to\regularbl$ the positive representation of $\boundedset$ on $\bl$ as constructed above. Then we shall say that \emph{$\rbs_\proj$
is generated by $\proj$}, and that \emph{$\proj$ is the generating positive spectral measure of $\rbs_\proj$ on $\sigmaset$}.
\end{definition}

We shall now concentrate on the question as to which positive representations of $\boundedset$ on Banach lattices have a generating positive spectral measure on $\sigmaset$. This is not always the case \cite{dejjia}. A positive representation that has a generating positive spectral measure is bounded according to Theorem~\ref{thm:generated_positive_representation}, but this is not necessarily a distinguishing feature, as is shown by the next result on automatic boundedness.

\begin{proposition}\label{prop:automatic_continuity_set}
Let $\sigmaset$ be a $\sigma$-algebra of subsets of a set $\set$, $\bl$ a
Banach lattice, and $\rbs \colon \boundedset \to\regularbl$ a positive representation. Then:
\begin{enumerate}
 \item $\rbs \colon \boundedset \to\regularblwithnorm$ is bounded, and $\norm{\rbs}=\norm{\rbs(\onefunctionset)}$.
 \item If $\bl$ is Dedekind complete, then $\rbs \colon \boundedset \to\regularblwithregnorm$ is bounded, and $\norm{\rbs}=\regnorm{\rbs}=\norm{\rbs(\onefunctionset)}$.
\end{enumerate}
\end{proposition}

\begin{proof}
As to (1), let $\phi\in\positive{\boundedset}$. Since $\zerofunctionset\leq\phi\leq\norm{\phi}\onefunctionset$ and $\rbs$ is positive, we have $0\leq\rbs(\phi)\leq\norm{\phi}\rbs(\onefunctionset)$. Hence $\norm{\rbs(\phi)}\leq\norm{\rbs(\onefunctionset)}\norm{\phi}$. For general $\phi=\positive{\phi}-\negative{\phi}\in\boundedset$, this implies that $\norm{\rbs(\phi)}\leq\norm{\rbs(\onefunctionset)}\,(\norm{\positive{\phi}}+ \norm{\negative{\phi}})\leq 2 \norm{\rbs(\onefunctionset)}\norm{\phi}$.  In particular, $\rbs \colon \boundedset \to\regularblwithnorm$ is bounded. For the statement concerning $\norm{\rbs}$ it is sufficient to show that $\norm{\rbs(\phi)}\leq\norm{\rbs(\onefunctionset)}\norm{\phi}$ for all $\phi$ in the dense subspace $\simpleset$ of $\boundedset$. As to this, we first note that the map $\mset\mapsto\rbs(\charfunc_\mset)$ is a positive operator-valued finitely additive  measure on $\sigmaset$. Proceeding as in the proof of Lemma~\ref{lem:total_variation_estimate}, one then sees that $\sum_{i=1}^n \abs{\pairing{\rct(\charfunc_{\mset_i})\elt,\dualelt}}\leq\pairing{\rct(\onefunctionset)\abs{\elt},\abs{\dualelt}}\leq\norm{\rct(\onefunctionset)}\norm{\elt}\norm{\dualelt}$ for all measurable disjoint partitions $X=\bigcup_{i=1}^n \mset_i$ of $\set$ and all $\condbothsp$. Taking the $\mset_i$ in $\phi=\sum_{i=1}^n \alpha_i\charfunc_{\mset_i}\in\boundedset$ to be a measurable disjoint partition of $\set$, this implies, as in the discussion preceding Theorem~\ref{thm:generated_positive_representation}, that

\begin{align*}
\norm{\rbs(\phi)}&=\sup\leb \abs{\pairing{\sum_{i=1}^n\alpha_i\rct(\charfunc_{\mset_i})\elt,\dualelt}}:\elt\in\bl,\,\norm{\elt}\leq 1,\,\dualelt\in\dualbl,\,\norm{\dualelt}\leq 1\rib \\
&\leq\sup\leb\sum_{i=1}^n \abs{\alpha_i}\abs{\pairing{\rct(\charfunc_{\mset_i})\elt,\dualelt}}:\elt\in\bl,\,\norm{\elt}\leq 1,\,\dualelt\in\dualbl,\,\norm{\dualelt}\leq 1\rib \\
&\leq\norm{\phi}\sup\leb\sum_{i=1}^n \abs{\pairing{\rct(\charfunc_{\mset_i})\elt,\dualelt}}:\elt\in\bl,\,\norm{\elt}\leq 1,\,\dualelt\in\dualbl,\,\norm{\dualelt}\leq 1\rib \\
&\leq\norm{\phi} \norm{\rbs(\onefunctionset)}.
\end{align*}

For part (2), let $\phi\in\boundedset$. Since $\abs{\phi}\leq\norm{\phi}\onefunctionset$, we have $\rbs(\abs{\phi})\leq\norm{\phi}\rbs(\onefunctionset)$. Therefore $0\leq\abs{\rbs(\phi)}\leq\rbs(\abs{\phi})\leq\norm{\phi}\rbs(\onefunctionset)$, which implies $\regnorm{\rbs(\phi)}=\norm{\abs{\rbs(\phi)}}\leq\norm{\norm{\phi}\rbs(\onefunctionset)}=\norm{\rbs(\onefunctionset)}\norm{\phi}$. This shows that $\rbs \colon \boundedset \to\regularblwithregnorm$ is bounded (which also follows from the automatic continuity of positive maps between Banach lattices), and that $\regnorm{\rbs}=\norm{\rbs(\onefunctionset)}$.
\end{proof}

\begin{remark}\label{rem:preview_2}
If $\bl$ is Dedekind complete, then it is well known that the operator norm and the regular norm coincide on the center $\centerbl$ of $\bl$. If $\rbs(\boundedset)\subset\centerbl$, the equality of $\norm{\rbs}$ and $\regnorm{\rbs}$, as asserted in Proposition~\ref{prop:automatic_continuity_set}, is then a priori clear. For general (not necessarily central) positive representations of
$\boundedset$ this equality may seem somewhat surprising, but also in this case more can be said \cite{dejjia}.
 \end{remark}

With the automatic boundedness available from Proposition~\ref{prop:automatic_continuity_set}, we can now give a description of the positive representations of $\boundedset$ that have a generating positive spectral measure on $\sigmaset$.

\begin{proposition}\label{prop:those_that_have_a_generating_spectral_measure}
Let $\sigmaset$ be a $\sigma$-algebra of subsets of a set $\set$, $\bl$ a
Banach lattice, and $\rbs \colon \boundedset \to\regularbl$ a positive representation.
Then the following are equivalent:
\begin{enumerate}
 \item $\rbs$ has a generating positive spectral measure on $\sigmaset$.
 \item If $(\mset_n)_{n=1}^{\infty}$ are pairwise disjoint elements of $\sigmaset$, and $\condeltpos$, then
 \[
 \rbs\lep\charfunc_{\cup_{n=1}^{\infty} \mset_n}\rip\elt=\sum_{n=1}^{\infty} \rbs(\charfunc_{\mset_n}) \elt,
 \]
 where the series converges in the norm topology.
 \item If $(\mset_n)_{n=1}^{\infty}$ are pairwise disjoint elements of $\sigmaset$, and $\condbothpossp$,  then
 \[\pairing{\rbs\lep\charfunc_{\cup_{n=1}^{\infty} \mset_n}\rip\elt,\dualelt}=\sum_{n=1}^{\infty}\pairing{ \rbs(\charfunc_{\mset_n}) \elt,\dualelt}.
 \]
\end{enumerate}
In that case, the generating positive spectral measure $\proj$ on $\sigmaset$ of $\rbs$ is given by $\proj(\mset)=\rbs(\charfunc_{\mset})$ \ulp $\mset\in\sigmaset$\urp.
\end{proposition}

\begin{proof}
Since the only possible generating positive spectral measure $\proj$ for $\rbs$ on $\sigmaset$ must be given by $\proj(\mset)=\rbs(\charfunc_{\mset})$ \ulp $\mset\in\sigmaset$\urp, it is clear that (1) implies (2). Clearly (2) implies (3). Assuming (3), we first observe that the equality in (3) then holds for all $\condboth$. We define $\proj(\mset)=\rbs(\charfunc_{\mset})$ $(\mset\in\sigmaset$). Then $P \colon \sigmaset\to\regularbl$ satisfies (1), (2), (3) and ($4^\prime$) in Definition~\ref{def:positive_spectral_measure} and Remark~\ref{rem:Pettis_consequence}, hence is a positive spectral measure on $\sigmaset$. It is clear that the positive representation $\rbs_\proj$ of $\boundedset$ on $\bl$ that is generated by $\proj$ agrees with $\rbs$ on $\simpleset$. Since both are bounded according to Proposition~\ref{prop:automatic_continuity_set}, $\rbs_\proj=\rbs$.
\end{proof}

For a fairly large practical class of lattices, the criterion for the existence of a generating positive spectral measure is particularly easy, and completely order-theoretical.

\begin{theorem}\label{thm:order_continuity_is_necessary_and_sufficient}
Let $\sigmaset$ be a $\sigma$-algebra of subsets of a set $\set$, $\bl$ a $\sigma$-Dedekind complete Banach lattice, and $\rbs \colon \sigmaset\to\regularbl$ a positive representation. If $\dualbl$ consists of $\sigma$-order continuous linear functionals only \ulp equivalently: if $\bl$ has $\sigma$-order continuous norm\urp, then the following are equivalent:
\begin{enumerate}
\item $\rbs$ has a generating positive spectral measure on $\sigmaset$.
\item $\rbs$ is $\sigma$-order continuous.
\end{enumerate}
\end{theorem}

\begin{proof}
We had already observed preceding Proposition~\ref{prop:order_continuity_is_necessary} that part (1) implies part (2), even without any further assumptions on $\bl$. For the converse implication, we verify the condition in part (3) of Proposition~\ref{prop:those_that_have_a_generating_spectral_measure}. In the pertinent notation, we let $\psi=\charfunc_{\bigcup_{n=1}^\infty\mset_n}$, and $\psi_N= \sum_{n=1}^N \charfunc_{\mset_n}$ ($N\geq 1$). Then $\psi_N\uparrow\psi$ in $\boundedset$, so $\rbs(\psi_N)\uparrow\rbs(\psi)$ in $\regularbl$ by the $\sigma$-order continuity of $\rbs$. By a straightforward modification of the proof of \cite[Theorem~1.18]{aliburpos}, the $\sigma$-Dedekind completeness of $\bl$ implies that $\rbs(\psi_N)\elt\uparrow\rbs(\psi)\elt$ for all $\condeltpos$. By the assumption on $\dualbl$ this implies that  $\pairing{\rbs(\psi_N)\elt,\dualelt}\uparrow\pairing{\rbs(\psi)\elt,\dualelt}$ for all $\condbothpossp$. That is, the condition in part (3) of Proposition~\ref{prop:those_that_have_a_generating_spectral_measure} is satisfied.
\end{proof}

\newpage

\begin{remark}\quad
\begin{enumerate}
\item Theorem~\ref{thm:order_continuity_is_necessary_and_sufficient} applies, in particular, to all Dedekind complete Banach lattices with order continuous
norm. A class of Banach lattice with a $\sigma$-order continuous norm that is not order continuous can be found in \cite[Example~7,
p.~46]{wnu}.
\item The equivalence involving the $\sigma$-order continuity can be found in \cite[p.~336]{zaa}.
\end{enumerate}
\end{remark}

\section{Positive $\contots$-representations generated by positive spectral
measures}\label{sec:existence_of_generating_regular_positive_spectral_measure}

In this section $\ts$ is a locally compact Hausdorff space. We shall investigate the relation between positive representations of $\contots$ on a Banach lattice $\bl$ and $\regularbl$-valued positive spectral measures on the Borel $\sigma$-algebra $\sigmats$ of $\ts$. One of the main results in this paper, Theorems~\ref{thm:generating_measure_KB}, is concerned with the existence of a generating regular positive spectral measure for such a representation. In most of the other results, such as Theorem~\ref{thm:construct_everything_from_rcts}, the existence of a generating (regular) positive spectral measure is merely assumed, and the relation between the representation and the spectral measure is studied.

We start with the following result on automatic boundedness of positive representations of $\contots$, in the same vein as Proposition~\ref{prop:automatic_continuity_set}.
\begin{proposition}\label{prop:automatic_continuity_ts}
Let $\ts$ be a locally compact Hausdorff space, $\bl$ a Banach lattice, and $\rct \colon \contots \to\regularbl$ a positive representation.
\begin{enumerate}
 \item  If $\ts$ is compact, and $\bl$ is not necessarily Dedekind complete, then:
\begin{enumerate}
 \item $\norm{\rct(\phi)}\leq\norm{\rct(\onefunctionts)}\,(\norm{\positive{\phi}}+ \norm{\negative{\phi}})$ \ulp $\phi\in\contots$\urp.
 \item $\rct \colon \contots \to\regularblwithnorm$ is bounded, and $\norm{\rct}\leq 2\norm{\rbs(\onefunctionts)}$.
\end{enumerate}
\item If $\ts$ is not necessarily compact, and $\bl$ is Dedekind complete, then the maps $\rct \colon \contots \to\regularblwithnorm$ and $\rct \colon \contots \to\regularblwithregnorm$ are both bounded, and $\norm{\rct}\leq\regnorm{\rct}$.
\item  If $\ts$ is compact, and $\bl$ is Dedekind complete, then $\rct\! \colon \contots\! \to\!\regularblwithnorm$ and $\rct \colon \contots \to\regularblwithregnorm$ are both bounded, and $\norm{\rct}=\regnorm{\rct}=\norm{\rct(\onefunctionts)}$.
\end{enumerate}
\end{proposition}

\begin{proof}
Part (1) follows as in the beginning of the proof of part (1) of Proposition~\ref{prop:automatic_continuity_set}.

As to part (2), if $\bl$ is Dedekind complete, then $\regularblwithregnorm$ is a Banach lattice. Hence the positive map $\rct \colon
\contots \to\regularblwithregnorm$ is automatically bounded. Using the contractivity of the inclusion map $\regularblwithregnorm\hookrightarrow\regularblwithnorm$ completes the proof of part (2).

As to part (3), it follows as in the proof of part (2) of Proposition~\ref{prop:automatic_continuity_set} that $\rct \colon
\contots \to\regularblwithregnorm$ is bounded and that $\regnorm{\rct}=\norm{\rct(\onefunctionts)}$. The contractivity of the inclusion map $\regularblwithregnorm\hookrightarrow\regularblwithnorm$ then implies that $\norm{\rbs}=\norm{\rbs(\onefunctionset)}$.

\end{proof}

At the time of writing we have no information for the case where $\ts$ is not compact and $\bl$ is not Dedekind complete, unless we assume that $\rct$ has a generating positive spectral measure (cf.\ Proposition~\ref{prop:norm_for_generated_representation}).

The following definition is the usual one.
\begin{definition}\label{def:regular_spectral_measure}
Let $\ts$ be a locally compact Hausdorff space with Borel $\sigma$-algebra $\sigmats$, $\bl$ a Banach lattice, and $\proj \colon \sigmats\to\regularbl$ a positive spectral measure. Then $\proj$ is \emph{regular} if the finite signed measure $\mustandardpair \colon \sigmats\to\R$, defined by $\mustandardpair(\mset)=\pairing{\proj(\mset)\elt,\dualelt}$ ($\mset\in\sigmats$), is a regular finite signed Borel measure for all $\condbothsp$.
\end{definition}

Interestingly enough, a regular positive spectral measure is also inner and outer regular on all elements of $\sigmats$ in a natural sense that is meaningful only in an ordered context.

\begin{proposition}\label{prop:regularity_of_the_spectral_measure}
Let $\ts$ be a locally compact Hausdorff space with Borel $\sigma$-algebra
$\sigmats$, $\bl$ a Banach lattice, and $\proj \colon \sigmats \to\regularbl$ a regular positive spectral measure. Then, for all $\mset\in\sigmats$,
\begin{enumerate}
\item $\proj(\mset)=\inf\leb \proj(V): V \textup{ open and }\mset\subset V\rib$ in $\regularbl$.
 \item $\proj(\mset)=\sup\leb \proj(K): K \textup{ compact and }K\subset\mset\rib$ in $\regularbl$.
 \end{enumerate}
\end{proposition}

We do not assume that $\bl$ is Dedekind complete, hence the existence of the infimum and supremum is not automatic.

\begin{proof}
If $\condbothpossp$, then the finite measure $\mustandardpair$ is a regular Borel measure by assumption. Part (2) of Theorem~\ref{thm:overview_theorem} shows that it is not only outer regular on all elements of $\sigmats$, but also inner regular on all elements of $\sigmats$. An appeal to Lemma~\ref{lem:inf_and_sup_msets} then finishes the proof.
\end{proof}

In view of the results in Section~\ref{sec:from_positive_spectral_measure_to_positive_representation}, the following usual definition is natural.

\begin{definition}
Let $\ts$ be a locally compact Hausdorff space with Borel $\sigma$-algebra $\sigmats$, $\bl$ a Banach lattice, and $\rct \colon \contots \to\regularbl$ a positive representation of $\contots$ on $\bl$. If $\proj \colon \sigmats \to\regularbl$ is a positive spectral measure on $\sigmats$, Section~\ref{sec:from_positive_spectral_measure_to_positive_representation} furnishes the positive representation $\rbt_\proj \colon \boundedts \to\regularbl$ of $\boundedset$ on $\bl$ that is generated by $\proj$. We say that \emph{$\proj$ generates $\rct$} if $\rct$ is the restriction of $\rct_\proj$ to $\contots$. If $\proj$ is a regular positive spectral measure on $\sigmats$ generating $\rct$, we say that $\rct$ \emph{has a generating regular positive spectral measure on $\sigmats$}.
\end{definition}

\begin{remark}\label{rem:uniqueness_and_automatic_regularity}
\quad
\begin{enumerate}
\item If $\rct$ has a generating regular positive spectral measure on $\sigmats$, then it is unique. This is immediate from \eqref{eq:weak_def} and the uniqueness statement in the Riesz representation theorem.
 \item For the sake of completeness, we note that every Borel measure is automatically both outer and inner regular on all elements of $\sigmats$ if every open subset of $\ts$ is $\sigma$-compact, i.e.\ if  it is the countable union of compact subsets of $\ts$ \cite[Theorem~7.8]{fol}. For such spaces (in particular, for all second countable spaces), every positive spectral measure on $\sigmats$ is therefore automatically regular.
\end{enumerate}
\end{remark}

We shall now establish an existence result for generating regular positive spectral measures.

Recall that a Banach lattice is a KB-space if every non-negative increasing norm bounded sequence is norm convergent. Reflexive Banach lattices and AL-spaces are KB-spaces \cite[p.~232]{aliburpos}. A KB-space has order continuous norm \cite[Theorem~2.4.2]{mey}, hence is Dedekind complete, and a Banach lattice $\bl$ is a KB-space if and only if $\bl$ is a band of $\bidualbl$ \cite[Theorem~4.60]{aliburpos}. The latter property is what makes our proof of the following theorem work.

\begin{theorem}\label{thm:generating_measure_KB}
Let $\ts$ be a locally compact Hausdorff space with Borel $\sigma$-algebra
$\sigmats$, $\bl$ a KB-space, and $\rct \colon \contots \to\regularbl$ a positive representation.

Then $\rct$ has a unique generating regular positive spectral measure $\proj$ on $\sigmats$.
\end{theorem}

\begin{proof}
The uniqueness of a generating regular positive spectral measure was already observed in the first part of Remark~\ref{rem:uniqueness_and_automatic_regularity}.

For its existence, we combine Theorem~\ref{thm:monotone_class_theorem} with ideas employed in the literature for unital representations of commutative $C^\ast$-algebras on Hilbert spaces \cite[Theorem~IX.1.4]{con} \cite[Theorem~12.22]{rud}, and for bounded unital representations of $\contots$ (where $\ts$ is compact) on Banach spaces \cite[Theorem~III.3]{ric}. The strategy is, as usual, to construct a (positive) representation $\erbttemp \colon \boundedts\to\regularbl$ that extends $\rct$, and then show that $\erbttemp$ has a generating regular (positive) spectral measure $\proj$ on $\sigmats$, so that one can actually write $\erbttemp=\erbt$.

 We start with the construction of $\erbttemp$.

 Since we had already observed that the KB-space $\bl$ is Dedekind complete, part (2) of Proposition~\ref{prop:automatic_continuity_ts} implies that $\rct \colon \contots \to\regularblwithnorm$ is bounded.

 Let $\condbothsp$, and consider the linear functional $\phi\mapsto\pairing{\rct(\phi)\elt,\dualelt}$ on $\contots$. We have
 \[
 \abs{\pairing{\rct(\phi)\elt,\dualelt}}\leq\norm{\rct}\norm{\phi}\norm{\elt}\norm{\dualelt}\quad(\phi\in\contots,\,\condbothsp).
 \]
 Consequently, this functional is bounded and has norm at most $\norm{\rct}\norm{\elt}\norm{\dualelt}$.
 The Riesz representation theorem furnishes a regular finite signed Borel measure $\mustandardpair$ such that
 \begin{equation}\label{eq:definition_of_measures}
 \pairing{\rct(\phi)\elt,\dualelt}=\intts\phi\,d\mustandardpair\quad(\phi\in\contots,\,\condbothsp).
 \end{equation}
 Moreover, $\norm{\mustandardpair}\leq\norm{\rct}\norm{\elt}\norm{\dualelt}$, and if $\condbothpossp$, then $\mustandardpair\geq 0$ as a consequence of the positivity of $\rct$. As a consequence of the uniqueness statement in the Riesz representation theorem, the map $(\elt,\dualelt)\mapsto\mustandardpair$ is bilinear. This implies that, for fixed $\phi\in\boundedts$, the form $\leh\,.\,,\,.\,\rih_\phi$ on $\bl\times\dualbl$, defined by
 \begin{equation*}
 \leh\standardpair\rih_\phi=\intts\phi\,d \mustandardpair\quad(\condbothsp),
 \end{equation*}
 is also bilinear. Moreover, $\abs{\leh\standardpair\rih_\phi}\leq \norm{\phi}\norm{\mustandardpair}\leq\norm{\phi}\norm{\pi}\norm{\elt}\norm{\dualelt}$. Hence $\leh\,.\,,\,.\,\rih_\phi$ is a bounded bilinear form on $\bl\times\dualbl$, and this implies that there exists a unique operator $\erbttemp(\phi)\in\bounded(\bl,\bidualbl)$ such that $\lea\erbttemp(\phi)\standardpair\ria=\leh\standardpair\rih_\phi$ for all $\condbothsp$. Hence
 \begin{equation}\label{eq:basic_relation_for_extension}
 \lea\erbttemp(\phi)\standardpair\ria=\intts\phi\,d \mustandardpair\quad(\phi\in\boundedts,\,\condbothsp).
 \end{equation}

 We shall now use the fact that $\bl$ is a band of $\bidualbl$ and Theorem~\ref{thm:monotone_class_theorem} to see that actually $\erbttemp(\phi)(\bl)\subset\bl$ (rather than $\bidualbl$) for all $\phi\in\boundedts$.

 To start, since $\mustandardpair\geq 0$ for $\condbothpossp$, \eqref{eq:basic_relation_for_extension} shows that $\erbttemp(\phi)\elt\geq 0$ if $\phi\geq 0$ and $\elt\in\blpos$. This implies that $\erbttemp(\phi)\in\bounded_{\textup{r}}(\bl,\bidualbl)$ ($\phi\in\boundedts$), and that $\erbttemp \colon \boundedts\to\bounded_{\textup{r}}(\bl,\bidualbl)$ is a (clearly linear) positive map. Furthermore, comparing \eqref{eq:definition_of_measures} and \eqref{eq:basic_relation_for_extension},
we see that $\erbttemp \colon \boundedts \to\bounded_{\textup{r}}(\bl,\bidualbl)$ extends $\rct \colon \contots \to\regularbl$, where we identify $\regularbl$ with the canonically corresponding subspace of $\bounded_{\textup{r}}(\bl,\bidualbl)$.
We let
 \[
 L=\leb\phi\in\boundedts : \erbttemp(\phi)(\bl)\subset \bl \rib.
 \]
 Since $\erbttemp$ extends $\rct$, we already know that $\contots\subset L$. We shall proceed to show that $L$ satisfies the two hypotheses in Theorem~\ref{thm:monotone_class_theorem}, so that actually $L=\boundedts$. As to the first hypothesis, we need to show that $\erbttemp(\charfunc_V)(\bl)\subset\bl$ for all open subsets $V$ of $\ts$. We may assume that $V\neq\emptyset$, and in that case we claim that
\begin{equation}\label{eq:sup_claim}
\erbttemp(\charfunc_V)=\sup \leb\erbttemp(\phi) : \phi\in\contcts,\,\supp \phi\subset V,\,\zerofunctionts\leq \phi\leq\onefunctionts\rib,
\end{equation}
where the right hand side is the supremum in the Dedekind complete Banach lattice $\bounded_{\textup{r}}(\bl,\bidualbl)$. Since $\erbttemp\geq 0$, it is clear that $\erbttemp(\charfunc_V)\geq\erbttemp(\phi)$ if $\phi\in\contcts$, $\supp \phi\subset V$, and $\zerofunctionts\leq \phi\leq\onefunctionts$. Suppose that $T\in\bounded_{\textup{r}}(\bl,\bidualbl)$ and that $T\geq\erbttemp(\phi)$ for all $\phi\in\contcts$ such that $\supp \phi\subset V$ and $\zerofunctionts\leq \phi\leq\onefunctionts$. Then from \eqref{eq:basic_relation_for_extension} we have, for all such $\phi$, and all $\condbothpossp$,
\begin{align*}
\pairing{T\elt,\dualelt}&\geq\pairing{\erbttemp(\phi)\elt,\dualelt}\\
&=\intts\phi\,d\mustandardpair.
\end{align*}
Therefore, for all $\condbothpossp$,
\begin{equation*}
\pairing{T\elt,\dualelt}\geq\sup\leb\intts\phi\,d\mustandardpair:\phi\in\contcts,\,\supp\phi\subset V,\,\zerofunctionts\leq\phi\leq\onefunctionts\rib.
\end{equation*}
 Now \eqref{eq:open_versus_functions} shows that the right hand side in this equation equals $\mustandardpair(V)$. We conclude that, for $\condbothpossp$,
 \begin{align*}
 \pairing{T\elt,\dualelt}&\geq\mustandardpair(V)\\
 &=\intts\charfunc_V\,d\mustandardpair\\
 &=\pairing{\erbttemp(\charfunc_V)\elt,\dualelt}.
 \end{align*}
Hence $T\geq\erbttemp(\charfunc_V)$, and our claim in \eqref{eq:sup_claim} has been established.

Since $\leb\phi : \phi\in\contcts,\,\supp\phi\subset V,\, \zerofunctionts\leq \phi\leq\onefunctionts\rib$ is directed upward and $\erbttemp$ is positive,  $\leb\erbttemp(\phi) : \phi\in\contcts,\,\supp \phi\subset V,\, \zerofunctionts\leq \phi\leq\onefunctionts\rib$ is a subset of $\bounded_{\textup{r}}(\bl,\bidualbl)$ that is directed upward. Hence its supremum $\erbttemp(\charfunc_V)$ can be determined pointwise on the positive cone of $\bl$ \cite[Theorem~1.19]{aliburpos}, and we conclude that
\[
\erbttemp(\charfunc_V)x=\sup\leb\erbttemp(\phi)\elt : \phi\in\contcts,\,\supp \phi\subset V,\, \zerofunctionts\leq \phi\leq\onefunctionts\rib\quad(\elt\in\blpos),
\]
where the supremum is in $\bidualbl$. However, since we had already observed that $\contots\subset L$, we know that $\leb\erbttemp(\phi)\elt : \phi\in\contcts,\,\supp \phi\subset V,\, \zerofunctionts\leq \phi\leq\onefunctionts\rib\subset\bl$ ($\elt\in\blpos$). Since $\bl$ is a band of $\bidualbl$, the supremum of this set in $\bidualbl$ is actually in $\bl$. We conclude that $\erbttemp(\charfunc_V)\elt\in\bl$ ($\elt\in\blpos$), and it follows that $\charfunc_V\in L$, as required in the first hypothesis in Theorem~\ref{thm:monotone_class_theorem}.

As to the second hypothesis, suppose that $\phi_n$ is a sequence of functions in $L$, and that $\phi$ is a bounded function on $\ts$ such that $0\leq \phi_n(\tselt)\uparrow \phi(\tselt)$ for all $\tselt\in\ts$. Then certainly $\phi\in\boundedts$, and we are left with showing that $\erbttemp(\phi)(\bl)\subset \bl$.  The proof of this bears some similarity to the above proof that $\erbttemp(\charfunc_V)(\bl)\in L$ for every open subset $V$ of $\ts$. We claim that $\erbttemp(\phi_n)\uparrow \erbttemp(\phi)$ in $\bounded_{\textup{r}}(\bl,\bidualbl)$.  Firstly, it is clear from the positivity of $\erbttemp$ that $\erbttemp(\phi_n)\uparrow$ and that $\erbttemp(\phi)\geq \erbttemp(\phi_n)$ for all $n$. Secondly, if $T\in\bounded_{\textup{r}}(\bl,\bidualbl)$ and $T\geq\erbttemp(\phi_n)$ for all $n$, then, for $\condbothpossp$, and $n\in\N$,
\begin{align}\label{eq:monotone_convergence}
\pairing{T\elt,\dualelt}&\geq\pairing{\erbttemp(\phi_n)\elt,\dualelt}\\
&=\intts\phi_n\,d\mustandardpair.\notag
 \end{align}
Applying the monotone convergence to \eqref{eq:monotone_convergence} yields $\pairing{T\elt,\dualelt}\geq \pairing{\erbttemp(\phi)\elt,\dualelt}$ ($\condbothpossp$); hence $T\geq \erbttemp(\phi)$. This establishes our claim.  A pointwise argument on $\blpos$ as for $\erbttemp(\charfunc_V)$ now implies that $\phi\in L$, as required in the second hypothesis in Theorem~\ref{thm:monotone_class_theorem}.

Theorem~\ref{thm:monotone_class_theorem} now shows that $\boundedts\subset L$, and we have finally established that $\erbttemp(\phi)(\bl)\subset\bl$
for all $\phi\in\boundedts$. Hence we can view $\erbttemp$ as a positive linear map $\erbttemp \colon \boundedts \to\boundedbl$  with codomain $\regularbl$ rather than $\bounded_{\textup{r}}(\bl,\bidualbl)$.

We shall now proceed to show that $\erbttemp$ is a positive representation of $\boundedts$, and that it has a generating regular positive spectral measure on $\sigmats$.

 For the multiplicativity of $\erbttemp$, we argue as follows. Let $\phi,\psi\in\contots$, and $\condbothsp$. Using \eqref{eq:basic_relation_for_extension} twice we have
 \begin{align*}
\intts \phi\psi\,d\mustandardpair&=\pairing{\rct(\phi\psi)\elt,\dualelt}\\
&=\pairing{\rct(\phi)\rct(\psi)\elt,\dualelt}\notag\\
&=\intts\phi\,d\mu_{\rct(\psi)\elt,\dualelt}\notag.
 \end{align*}
 Lemma~\ref{lem:then_for_bounded_borel} then shows that
 \begin{equation}\label{eq:first_step}
\intts \phi\psi\,d\mustandardpair=\intts\phi\,d\mu_{\rct(\psi)\elt,\dualelt}\notag,
 \end{equation}
for all $\phi\in\boundedts,\,\psi\in\contots$ and $\condbothsp$.
This implies, using \eqref{eq:basic_relation_for_extension} in the second and fourth step, that
\begin{align*}
\intts \phi\psi\,d\mustandardpair&=\intts\phi\,d\mu_{\rct(\psi)\elt,\dualelt}\\
&=\pairing{\erbttemp(\phi)\rct(\psi)\elt,\dualelt}\notag\\
&=\pairing{\rct(\psi)\elt,\adjoint{\erbttemp(\phi)}\dualelt}\notag\\
&=\intts\psi\,d\mu_{\elt,\adjoint{\erbttemp(\phi)}\dualelt},\notag
\end{align*}
for all $\phi\in\boundedts,\,\psi\in\contots$ and $\condbothsp$. Lemma~\ref{lem:then_for_bounded_borel} now shows that
\begin{equation*}
\intts \phi\psi\,d\mustandardpair=\intts\psi\,d\mu_{\elt,\adjoint{\erbttemp(\phi)}\dualelt},\notag
\end{equation*}
for all $\phi,\psi\in\boundedts$ and $\condbothsp$. Using \eqref{eq:basic_relation_for_extension} in the second step, we then see that
\begin{align}\label{eq:fourth_step}
 \intts \phi\psi\,d\mustandardpair&=\intts\psi\,d\mu_{\elt,\adjoint{\erbttemp(\phi)}\dualelt}\\
 &=\pairing{\erbttemp(\psi)\elt,\adjoint{\erbttemp(\phi)}\dualelt}\notag\\
 &=\pairing{\erbttemp(\phi)\erbttemp(\psi)\elt,\dualelt},\notag
\end{align}
for all $\phi,\psi\in\boundedts$ and $\condbothsp$. On the other hand, \eqref{eq:basic_relation_for_extension} shows that
\[
 \intts \phi\psi\,d\mustandardpair=\pairing{\erbttemp(\phi\psi)\elt,\dualelt},
\]
for all $\phi,\psi\in\boundedts$ and $\condbothsp$. Comparing this with \eqref{eq:fourth_step}, we conclude that $\erbttemp(\phi\psi)=\erbttemp(\phi)\erbttemp(\psi)$ for all $\phi,\psi\in\boundedts$. Hence $\erbttemp \colon \boundedts \to\regularbl$ is a positive representation of $\boundedts$ on $\bl$.

To show that $\erbttemp$ has a generating positive spectral measure on $\sigmats$, we shall verify condition (3) in Proposition~\ref{prop:those_that_have_a_generating_spectral_measure}. Let $\condbothpossp$, and let $(\mset_n)_{n=1}^{\infty}$ be pairwise disjoint elements of $\sigmaset$. We must show that $\pairing{\erbttemp\lep\charfunc_{\cup_{n=1}^{\infty}
\mset_n}\rip\elt,\dualelt}=\sum_{n=1}^{\infty}\pairing{ \erbttemp(\charfunc_{\mset_n})
\elt,\dualelt}$. Using the monotone convergence theorem in the second step, we see that
\begin{align*}
\pairing{\erbttemp\lep\charfunc_{\cup_{n=1}^{\infty}
\mset_n}\rip\elt,\dualelt}&=\intts \charfunc_{\cup_{n=1}^{\infty}\mset_n}\,d\mustandardpair\\
&=\sum_{n=1}^\infty \intts \charfunc_{\mset_n}\,d\mustandardpair\\
&=\sum_{n=1}^{\infty}\pairing{ \erbttemp(\charfunc_{\mset_n})
\elt,\dualelt}.
\end{align*}
Hence we conclude from Proposition~\ref{prop:those_that_have_a_generating_spectral_measure} that $\erbttemp$ has a generating positive spectral $\proj$ measure on $\sigmats$ that is given by $\proj(\mset)=\erbttemp(\charfunc_\mset)$.  In order to conclude that $\proj$ is regular, we consider its associated signed measures, denoted temporarily by $\mustandardpair^\proj$ ($\condbothsp$). For $\mset\in\sigmats$, we have, using \eqref{eq:basic_relation_for_extension} in the third step,
\begin{align*}
\mustandardpair^\proj(\mset)&=\pairing{\proj(\mset)\elt,\dualelt}\\&=\pairing{\erbttemp(\charfunc_\mset)\elt,\dualelt}\\
&=\intts \charfunc_\mset\,d\mustandardpair\notag\\
&=\mustandardpair(\mset).\notag
\end{align*}
Hence $\mustandardpair^\proj=\mustandardpair$ ($\condbothsp$), which is known to be regular.
\end{proof}

\begin{remark}\label{rem:alternative_approaches}
The above proof of Theorem~\ref{thm:generating_measure_KB} makes it clear what is the essential feature of a KB-space $E$ in this context: it is a band of $\bidualbl$. Since we are concerned with positive representations, such an order-theoretical property and its ensuing role in the proof seem natural. It should be noted, however, that it is also possible to obtain the existence part in Theorem~\ref{thm:generating_measure_KB} using various Banach space characterisation of KB-spaces, along the following lines.

We first note that, by part (2) of Proposition~\ref{prop:automatic_continuity_ts}, the Dedekind completeness of the KB-space $\bl$ implies that
$\rct \colon \contots\!\to\!\boundedblwithnorm$ is bounded. Even if $\ts$ is compact, so that $\contots$ has a unit, we can still consider the augmented algebra $\contots_1=\R\times\contots$ of $\contots$ (which is the usual unitisation of $\contots$ if $\ts$ is not compact), with
norm $\norm{(\lambda,\phi)}=\abs{\lambda}+\norm{\phi}$ ($\lambda\in\R$, $\phi\in\contots$). Then the representation $\rct_\infty \colon
\contots_1 \to\boundedblwithnorm$, defined by $\rct_\infty((\lambda,\phi))=\lambda\,\idmap_\bl + \rct(\phi)$ ($\lambda\in\R,\phi\in\contots$),
is a unital representation of $\contots_1$ on $\bl$ that is also bounded.

Next, also if $\ts$ is already compact, we let $\ts_\infty=\ts\cup\leb\infty\rib$ be the one point compactification of $\ts$. The algebras $\conto{\ts_\infty}$ and $\contots_1$
are canonically isomorphic as abstract algebras, and although the pertinent isomorphism is not necessarily isometric, it is still a
linear homeomorphism. Hence $\rct_\infty$ can be viewed as a bounded unital representation of $\conto{\ts_\infty}$ on $\bl$ that extends
$\rct$.

There are now various ways to proceed:
\begin{enumerate}
\item A Banach lattice $E$ is a KB-space if and only if it does not contain a subspace linearly homeomorphic to $c_0$ \cite[Theorem~7.1]{wnu}. For Banach spaces that do not contain such a subspace, \cite[Theorem~III.4]{ric} provides a generating regular spectral measure for bounded unital representations of $\contots$-spaces (for compact $\ts$) on them. In particular, this applies to the bounded unital representation $\rct_\infty$ of $\conto{\ts_\infty}$ on $\bl$.
\item A Banach lattice $E$ is a KB-space if and only if it is weakly sequentially complete \cite[Theorem~7.1]{wnu}. For weakly sequentially complete Banach spaces, \cite[Theorem~XVII.2.5]{dunsch3} provides a generating regular spectral measure for bounded unital representations of $\contots$-spaces (for compact $\ts$) on them. Again, this applies, in particular, to the bounded unital representation $\rct_\infty$ of $\conto{\ts_\infty}$ on $\bl$.
\item Since $E$ does not contain a subspace linearly homeomorphic to $c_0$, every continuous linear map from $C_0(X)$ (for compact $X$) into $E$ is weakly compact \cite[Theorem~I.14]{ric}. In particular, for all $\elt\in\bl$, the map $f\mapsto\rct_\infty(f)x$ ($f\in\conto{\ts_\infty}$) is weakly compact. Hence \cite[Theorem~3]{orh} applies, and provides a generating spectral measure for $\rct_\infty$.
\end{enumerate}
Once we have concluded that there exists a regular spectral measure $\proj_\infty$ on the Borel $\sigma$-algebra $\sigmats_\infty$ of $\ts_\infty$ such that
\begin{equation}\label{eq:compactification_weak_def}
\pairing{\rct_\infty(\phi)\elt,\dualelt}=\int_{\ts_\infty} \phi\,d\mustandardpair^{P_\infty}\quad(\phi\in \conto{\ts_\infty},\,\condbothsp),
\end{equation}
where $\mustandardpair^{P_\infty}(\mset)=\pairing{\proj_\infty(\mset)\elt,\dualelt}$ ($\mset\in\sigmats_\infty,\,\condbothsp)$, we define
$\proj \colon \sigmats\to\regularbl$ by $\proj(\mset)=\proj_\infty(\mset)$ ($\mset\in\sigmats\subset\sigmats_\infty$).
It is routine to check that $\proj$ is a regular spectral measure on $\sigmats$, and we let $\mustandardpair$ ($\condbothsp$) denote the usual associated
regular finite signed Borel measures on $\sigmats$. If $\phi\in\contots$, then $\phi(\infty)=0$, so \eqref{eq:compactification_weak_def}
implies that
\begin{equation}\label{eq:compactification_back}
\pairing{\rct(\phi)\elt,\dualelt}=\intts \phi\,d\mustandardpair\quad(\phi\in\contots,\,\condbothsp).
\end{equation}
Since $\rct$ is positive, the order statement in the Riesz representation theorem implies that all measures $\mustandardpair$ ($\condbothpossp$)
are positive, which shows that $\proj$ takes its values in the positive projections on $\bl$. Comparison of \eqref{eq:compactification_back} with \eqref{eq:weak_def}
yields that the positive representation $\rbt_\proj \colon \boundedts\to\regularbl$ that is generated by $\proj$ restricts to $\rct$
on $\contots$, as required. This concludes the alternative Banach space proof of Theorem~\ref{thm:generating_measure_KB}.

The reader may verify that each of the three aforementioned existence results for generating spectral measures uses a substantial amount of theory of Boolean algebras of projections, as well as general Banach space theory. Since, in addition, these existence results can only be applied once the pertinent Banach space property of a KB-space has been established, we feel that our proof of the existence part in Theorem~\ref{thm:generating_measure_KB}, that  exploits only order-theoretical properties, is not only more natural in this context, but also considerably simpler than the above alternatives.
\end{remark}

\begin{remark}
If $\ts$ is compact, additional existence results for regular positive spectral measures generating so-called $R$-bounded unital positive
representations of $\contots$ can be obtained using \cite[Proposition~2.17]{pagric}.
\end{remark}

After these remarks regarding Theorem~\ref{thm:generating_measure_KB}, we continue with a consequence thereof that is clear.

\begin{corollary}\label{cor:bijection}
Let $\ts$ be a locally compact Hausdorff space with Borel $\sigma$-algebra
$\sigmats$, and $\bl$ a KB-space.
Then the map $\proj\to\rbt_\proj\rest{\contots}$, sending an $\regularbl$-valued positive spectral measure on $\sigmats$ to the restriction of $\rbt_\proj \colon \boundedts \to
\regularbl$ to $\contots\subset\boundedts$, restricts to a bijection between the $\regularbl$-valued regular positive spectral measures on $\sigmats$ and the positive representations of $\contots$ on $\bl$.

If $\ts$ is compact, then $\rbt_\proj\rest{\contots}$ is
unital if and only if $\proj$ is unital.
\end{corollary}

We shall now concentrate on the implications of the existence of a generating (regular) positive spectral measure for $\rct$. To start with, we have the following result. Note that it also covers the `missing' case in Proposition~\ref{prop:automatic_continuity_ts}, but only under the hypothesis of the existence of a generating positive spectral measure.

\begin{proposition}\label{prop:norm_for_generated_representation}
Let $\ts$ be a locally compact Hausdorff space with Borel $\sigma$-algebra
$\sigmats$, $\bl$ a Banach lattice, and $\rct \colon \contots \to\regularbl$ a positive representation. Suppose $\rct$ has a generating positive spectral measure $\proj$ on $\sigmats$, and let $\rbt_\proj \colon \boundedts\to\regularbl$ denote the generated bounded positive representation of $\boundedts$ on $\bl$ extending $\rct$. Then:
\begin{enumerate}
 \item $\rct \colon \contots\to\regularblwithnorm$ is bounded, and $\norm{\rct}\leq\norm{\rct_\proj}=\norm{\maxprojts}$.
 \item If $\proj$ is regular, then $\norm{\rct}=\norm{\rct_\proj}=\maxprojset$.
 \item If $\proj$ is regular and $\bl$ is Dedekind complete, then the maps
$\rct \colon \contots \to\regularblwithnorm$, $\rct \colon \contots \to\regularblwithregnorm$,
 $\erbt \colon \boundedts \to\regularblwithnorm$, and $\erbt \colon \boundedts \to\regularblwithregnorm$
are all bounded, and $\norm{\rct}=\regnorm{\rct}=\norm{\erbt}=\regnorm{\erbt}=\norm{\maxprojts}$.
 \end{enumerate}
\end{proposition}

\begin{proof}
We know from part (1) of Theorem~\ref{thm:generated_positive_representation} that $\norm{\rbt_\proj}=\norm{\maxprojts}$. Certainly the restriction $\rct$ of $\rbt_\proj$ to $\contots$ is also bounded, and $\norm{\rct}\leq\norm{\rbt_\proj}$; this establishes part (1). For the converse inequality that is needed for part (2) if $\proj$ is regular, we use \eqref{eq:weak_def} and the isometry statement in the Riesz representation theorem to see that
\begin{align*}
\Vert &\pi_P\Vert =\\
&=\sup\leb \lev\pairing{\rbt(\phi)\elt,\dualelt}\riv : \phi\in\boundedts,\,\norm{\phi}\leq 1,\,\condelt,\,\norm{\elt}\leq 1,\,\conddualelt,\,\norm{\dualelt}\leq 1\rib\\
&= \sup\leb \lev\intts\phi\,d\mustandardpair\riv: \phi\in\boundedts,\,\norm{\phi}\leq 1,\,\condelt,\,\norm{\elt}\leq 1,\,\conddualelt,\,\norm{\dualelt}\leq 1\rib\\
&\leq \sup\leb \norm{\mustandardpair}: \condelt,\,\norm{\elt}\leq 1,\,\conddualelt,\,\norm{\dualelt}\leq 1\rib\\
&=\sup\leb \lev\intts\phi\,d\mustandardpair\riv: \phi\in\contots
,\norm{\phi}\leq 1,\,\condelt,\norm{\elt}\leq 1,\,\conddualelt,\,\norm{\dualelt}\leq 1\!\rib\\
&=\sup\leb \lev\pairing{\rbt(\phi)\elt,\dualelt}\riv : \phi\in\contots,\,\norm{\phi}\leq 1,\,\condelt,\norm{\elt}\leq 1,\,\conddualelt,\,\norm{\dualelt}\leq 1\rib\\
&=\norm{\rct}.
\end{align*}

As to (3), for Dedekind complete $\bl$ the positive maps $\rct\colon \contots\to\regularblwithregnorm$ and $\rbt_\proj \colon \boundedts \to\regularblwithregnorm$ between Banach lattices are bounded, and then the contractivity of the inclusion map $\regularblwithregnorm\hookrightarrow\boundedblwithnorm$ implies that the other two maps are bounded as well. Furthermore, parts (7)(a) and (1) of Theorem~\ref{thm:generated_positive_representation} show that
$\regnorm{\erbt}=\norm{\erbt}=\norm{\maxprojts}$.
We also know from part (2) of Proposition~\ref{prop:automatic_continuity_ts} that $\norm{\pi}\leq\regnorm{\pi}$. In addition, we have $\regnorm{\pi}\leq\regnorm{\erbt}$, since $\erbt$ extends $\rct$. If $\proj$ is regular, then we have already established in part (2) that $\norm{\rbt_\proj}=\norm{\rct}$. Combining all this, we see that, for regular $\proj$ and Dedekind complete $\bl$,
\[
\regnorm{\erbt}=\norm{\erbt}=\norm{\rct}\leq\regnorm{\rct}\leq\regnorm{\erbt},
\]
and the proof is complete.
\end{proof}

We collect a few further consequences (some of them of course familiar from the non-ordered context) of the existence of a generating positive spectral measure for $\rct \colon \contots\to\regularbl$ in our next result.

\begin{proposition}\label{prop:usual_consequences}
Let $\ts$ be a locally compact Hausdorff space with Borel $\sigma$-algebra
$\sigmats$, $\bl$ a Banach lattice, and $\rct \colon \contots \to\regularbl$ a positive representation. Suppose $\rct$ has a generating positive spectral measure $\proj$ on $\sigmats$, and let $\rbt_\proj \colon \boundedts\to\regularbl$ denote the generated bounded positive representation of $\boundedts$ on $\bl$ extending $\rct$.
Then:
\begin{enumerate}
\item If $\phi\in\contots$, then there is a sequence of linear combinations of elements of $\proj(\sigmats)$ that converges to $\rct(\phi)$ in $\regularblwithnorm$. If $\bl$ is Dedekind complete, then there exists such a sequence converging to $\rct(\phi)$ in $\regularblwithregnorm$. If $\phi\in\positive{\contots}$, then the coefficients occurring in these linear combinations can be taken non-negative.
\item If $\phi\in\contots$, and $\lep\phi\rip_{n=1}^\infty\subset\contots$
is a norm bounded sequence such that $\lim_{n\to\infty}\phi_n(\tselt)=\phi(\tselt)$
for all $\tselt\in\ts$, then $\rct(\phi)=\textup{WOT-}\lim_n\rct(\phi_n)$.
\item If $\phi\in\contots$ and $\lep\phi\rip_{n=1}^\infty\subset\contots$
is a norm bounded sequence such that $\phi_n(\tselt)\uparrow\phi(\tselt)$ for all $\tselt\in\ts$, then $\rct(\phi_n)\uparrow\rct(\phi)$ in $\regularbl$.
\end{enumerate}
If $\proj$ is regular, then:
\begin{enumerate}
\item[(4)] The commutants $\commutant{\proj(\sigmats)}$, $\commutant{\erbt(\simplets)}$,
$\commutant{\erbt(\boundedts)}$, and $\commutant{\rct(\contots)}$ in
$\boundedbl$ are equal. Consequently, $\proj(\sigmats)\subset\bicommutant{\rct(\contots)}$.
\end{enumerate}
\end{proposition}

\begin{proof}
For part (1), we take a sequence of simple functions converging uniformly to $\phi$ in $\boundedts$ and apply parts (1) and (7)(a) of Theorem~\ref{thm:generated_positive_representation}.

Part (2) is a specialisation of part (4) of Theorem~\ref{thm:generated_positive_representation}.

Part (3) is a specialisation of Proposition~\ref{prop:order_continuity_is_necessary}.

As to part (4), from part (6) of Theorem~\ref{thm:generated_positive_representation} we already know that the commutants $\commutant{\proj(\sigmats)}$, $\commutant{\erbt(\simplets)}$, and $\commutant{\erbt(\boundedts)}$ in $\boundedbl$ are equal. We shall show that $\commutant{\rct(\contots)}=\commutant{\proj(\sigmats)}$. Let $T\in\boundedbl$. Then $T\in\commutant{\rct(\contots)}$ if and only if $\pairing{T\rct(\phi)\elt,\dualelt}=\pairing{\rct(\phi)T\elt,\dualelt}$ for all $\phi\in\contots$ and $\condbothsp$. Now note that
\begin{align*}
 \pairing{T\rct(\phi)\elt,\dualelt}&=\pairing{\rct(\phi)\elt,\adjoint{T}\dualelt}\\
 &=\intts \phi\,d\mu_{\elt,\adjoint{T}\dualelt},
\end{align*}
and that
\begin{align*}
\pairing{\rct(\phi)T\elt,\dualelt}&=\intts\phi\,d\mu_{T\elt,\dualelt}.
\end{align*}
Thus $T\in\commutant{\rct(\contots)}$ if and only if
\begin{equation*}
 \intts \phi\,d\mu_{\elt,\adjoint{T}\dualelt}=\intts\phi\,d\mu_{T\elt,\dualelt},
\end{equation*}
for all $\phi\in\contots$ and $\condbothsp$. By the uniqueness statement in the Riesz representation theorem, this is the case if and only if $\mu_{\elt,\adjoint{T}\dualelt}=\mu_{T\elt,\dualelt}$ for all $\condbothsp$. That is, if and only if $\pairing{\proj(\mset)\elt,\adjoint{T}\dualelt}=\pairing{\proj(\mset)T\elt,\dualelt}$ for all $\mset\in\sigmats,\condboth$.  This, in turn, is equivalent to $T\in\commutant{\proj(\sigmats)}$.

The folklore final statement is immediate from $\proj(\sigmaset)\subset\bicommutant{\proj(\sigmaset)}=\bicommutant{\rct(\contots)}$.
\end{proof}

We conclude by showing how the generating regular positive spectral measure $\proj$ of $\rct$, if it exists, can be determined directly from $\rct$ in terms of the ordering on $\regularbl$. By the first part of Proposition~\ref{prop:regularity_of_the_spectral_measure}, it is sufficient to know $\proj(V)$ for all open subsets $V$ of $\ts$, and part (1) of the next result shows how $\proj(V)$ can be found from $\rct(\contots)$. Likewise, the second part of Proposition~\ref{prop:regularity_of_the_spectral_measure} shows that it is sufficient to know $\proj(K)$ for all compact subsets $K$ of $\ts$, and part (2) of the next result shows how to retrieve these from $\rct(\contots)$.

As in similar previous results, we do not assume that $\bl$ is Dedekind complete, hence the existence of the various suprema and infima in $\regularbl$ is not automatic.

\begin{theorem}\label{thm:construct_everything_from_rcts}
Let $\ts$ be a locally compact Hausdorff space with Borel $\sigma$-algebra
$\sigmats$, $\bl$ a Banach lattice, and $\rct \colon \contots \to\regularbl$ a positive representation. Suppose $\rct$ has a generating regular positive spectral measure $\proj$ on $\sigmats$.
\begin{enumerate}
\item Let $V$ be an open subset of $\ts$. Then
\begin{equation*}
\proj(V)=\sup \leb\rct(\phi) : \phi\in\contcts,\,\supp \phi\subset V,\, \zerofunctionts\leq \phi\leq\onefunctionts\rib.
\end{equation*}
\item Let $K$ be a compact subset of $\ts$. Then
\begin{equation*}
\proj(K)=\inf \leb\rct(\phi) : \phi\in\contcts,\,\phi\geq\charfunc_K\rib.
\end{equation*}
\item In addition to the expression for $\maxprojset$ as obtained from part \ulp 1\urp, we also have
\begin{equation*}
\maxprojts=\sup \leb\rct(\phi) : \phi\in\contots, \zerofunctionts\leq \phi\leq\onefunctionts\rib.
\end{equation*}
\item Let $V$ be an open subset of $\ts$. Then $V$ is $\sigma$-compact if and only if there exists a sequence $\lep\phi_n\rip_{n=1}^\infty$ in $\contcts$ such that $\supp\phi_n\subset V$  \ulp$n\geq 1$\urps and $\sup_{n}\phi_n=\charfunc_V$ in $\boundedts$.

In that case, there exists a sequence $\lep\phi_n\rip_{n=1}^\infty$ \!\!\!\ in $\contcts$ such that $\supp\phi_n\! \subset V$ and $\zerofunctionts\leq\phi_n\leq\onefunctionts$ \ulp$n\geq 1$\urp, and $\phi_n\uparrow\charfunc_V$ in $\boundedts$.

For any norm bounded sequence $(\phi)_{n=1}^\infty$ in $\contcts$ such that $\phi_n\uparrow\charfunc_V$ in $\boundedts$, we have  $\rbs(\phi_n)\uparrow\proj(V)$ in $\regularbl$ and $\proj(V)=\textup{WOT-}\lim_n\phi_n$.

\item $\ts$ is $\sigma$-compact if and only if there exists a sequence $\lep\phi_n\rip_{n=1}^\infty$ in $\contots$ such that $\zerofunctionts\leq\phi_n\leq\onefunctionts$ \ulp$n\geq 1\urp$ and $\sup_{n}\phi_n=\onefunctionts$ in $\boundedts$.

In that case, there exists a sequence $\lep\phi_n\rip_{n=1}^\infty$ in $\contcts$ such that  $\zerofunctionts\leq\phi_n\leq\onefunctionts$ \ulp$n\geq 1$\urp, and $\phi_n\uparrow\onefunctionts$ in $\boundedts$.

For any norm bounded sequence $\lep\phi_n\rip_{n=1}^\infty$ in $\contots$ such that $\phi_n\uparrow\onefunctionts$ in $\boundedts$, we have  $\rbs(\phi_n)\uparrow\maxprojts$ in $\regularbl$ and $\maxprojts=\textup{WOT-}\lim_n\phi_n$.
\end{enumerate}
\end{theorem}

\begin{proof}
Let $\rct_\proj \colon \boundedts\to\regularbl$ denote the positive representation of $\boundedts$ on $\bl$ that is generated by $\proj$ and that extends $\rct$, with associated regular finite signed Borel measures $\mustandardpair$ ($\condbothsp$). Starting with part (1), if $\phi\in\contcts$, $\supp\phi\subset V$, and $\zerofunctionts\leq\phi\leq\onefunctionts$, then $\rct(\phi)=\erbt(\phi)\leq\erbt(\charfunc_V)=\proj(V)$. Hence $\proj(V)$ is an upper bound for
\[
\leb\rct(\phi) : \phi\in\contcts,\,\supp \phi\subset V,\, \zerofunctionts\leq \phi\leq\onefunctionts\rib.
\]
If $T\in\regularbl$ is also an upper bound for this set, then, for all $\phi\in\contcts$ with $\supp\phi\subset V$ and $\zerofunctionts\leq\phi\leq\onefunctionts$, and all $\condbothpossp$, we have
\begin{align*}
\pairing{T\elt,\dualelt}&\geq\pairing{\rct(\phi)\elt,\dualelt}\\
&=\intts\phi\,d\mustandardpair,
\end{align*}
where \eqref{eq:weak_def} was used.
Therefore, for all $\elt\in\blpos,\dualelt\in\dualblpos$,
\begin{equation*}
\pairing{T\elt,\dualelt}\geq\sup\leb\intts\phi\,d\mustandardpair:\phi\in\contcts,\,\supp\phi\subset V,\, \zerofunctionts\leq\phi\leq\onefunctionts\rib.
\end{equation*}
Since by \eqref{eq:open_versus_functions} the right hand side in this equation equals $\mustandardpair(V)=\pairing{\proj(V)\elt,\dualelt}$, we conclude that $\pairing{T\elt,\dualelt}\geq\pairing{\proj(V)\elt,\dualelt}$ for all $\condbothpossp$.
Hence $T\geq\proj(V)$.

The proof of part (2) is similar, based on \eqref{eq:compact_versus_functions}.

For part (3), the same line of reasoning shows that $\maxprojset\geq\rct(\phi)$ for all $\phi\in\contots$ with $\zerofunctionts\leq \phi\leq\onefunctionts$. If $T\geq\rct(\phi)$ for all such $\phi$, then, for all $\condbothpossp$, we find
\begin{equation*}\label{eq:isometry_part}
\pairing{T\elt,\dualelt}\geq\sup\leb\intts\phi\,d\mustandardpair:\phi\in\contots,\, \zerofunctionts\leq\phi\leq\onefunctionts\rib.
\end{equation*}
But the right hand side is the norm of the positive functional $\phi\mapsto\intts\phi\,d\mustandardpair$ on $\contots$, which by the isometry statement in the Riesz representation theorem equals $\norm{\mustandardpair}=\mustandardpair(\ts)=\pairing{\maxprojts\elt,\dualelt}$. Therefore $T\geq\maxprojset$.

For part (4), if such a sequence exists, then $V=\bigcup_{n}\leb\tselt\in\ts : \phi_n(\tselt)>0\rib$ is countable union of $\sigma$-compact subsets of $\ts$, hence $\sigma$-compact. Conversely, if $V$ is $\sigma$-compact, then we may assume that $V=\bigcup_n K_n$ where $K_n\subset K_{n+1}$ for all $n$. By \cite[Corollary~4.32]{fol} we can choose $\psi_n\in\contcts$ such that $\zerofunctionts\leq\psi_n\leq\onefunctionts$, $\psi_n(\tselt)=1$ for $\tselt\in K_n$ and $\supp(\psi_n)\subset V$. Let $\phi_n=\bigvee_{k=1}^n\psi_k$. Then the sequence $\lep\phi_n\rip_{n=1}^\infty$ is as required. An appeal to Proposition~\ref{prop:order_continuity_is_necessary} concludes the proof of part (4).

The proof of part (5) is similar to that of part (4).
\end{proof}

\subsection*{Acknowledgments}
The authors thank Ben de Pagter for helpful discussions, and the anonymous referees for their remarks and suggestions.


\begin{thebibliography}{99}

\bibitem{abrali}
Y.A.\ Abramovich and C.D.\ Aliprantis, \emph{An invitation to operator theory}, American Mathematical Society, Providence, Rhode Island, 2002.

\bibitem{abrarekit}
 Y.A.Abramovich, E.L.\ Arenson and A.K.\ Kitover, \emph{Banach $C(K)$-modules and operators preserving disjointness}, Longman Scientific \& Technical, Harlow, UK, 1992.

\bibitem{aliburpri}
C.D.\ Aliprantis and O.\ Burkinshaw, \emph{Principles of real analysis}, $3^{\textup {rd}}$ ed., Academic Press, San Diego, 1998.

\bibitem{aliburpro}
C.D.\ Aliprantis and O.\ Burkinshaw, \emph{Problems in real analysis. A workbook with solutions}, $2^{\textup {nd}}$ ed., Academic
Press, San Diego, 1999.

\bibitem{aliburpos}
 C.D.\ Aliprantis and O.\ Burkinshaw, \emph{Positive operators}, Springer, Dordrecht, 2006.

\bibitem{con}
J.B.\ Conway, \emph{A course in functional analysis}, Springer, New York, 2007.

\bibitem{dieuhl}
J.\ Diestel, J.J.\ Uhl, \emph{Vector measures}, American Mathematical Society, Providence, Rhode Island, 1977.

\bibitem{dunsch1}
N.\ Dunford and J.T.\ Schwartz, \emph{Linear operators, Part I. General theory}, Interscience Publishers, New York, 1958.

\bibitem{dunsch2}
N.\ Dunford and J.T.\ Schwarz, \emph{Linear operators. Part II. Spectral theory}, Interscience Publishers, New York, 1963.

\bibitem{dunsch3}
N.\ Dunford and J.T.\ Schwarz, \emph{Linear operators. Part III. Spectral operators}, Interscience Publishers, New York, 1971.

\bibitem{fol}
G.B.\ Folland, \emph{Real analysis}, $2^{\textup {nd}}$ ed., John Wiley, New York, 1999.

\bibitem{dejjia}
M.\ de Jeu and X.\ Jiang, \emph{Positive representations of $\contots$. II.}, in preparation.

\bibitem{kai}
 S.~Kaijser, \emph{Some representation theorems for Banach lattices}, Ark.\ Math.\ {\bf 16} (1978), 179-193.

\bibitem{mey}
P.\ Meyer-Nieberg, \emph{Banach lattices}, Springer, Berlin, 1991.

\bibitem{orh}
M.\ Orhon, \emph{Boolean algebras of commuting projections}, Math.\ Z.\ {\bf 183} (1983), 531--537.

\bibitem{pagric}
 B.\ de Pagter and W.J.\ Ricker, \emph{$C(K)$-representations and $R$-boundedness}, J.\ London Math.\ Soc.\ (2) \textbf{76} (2007), 498-512.

\bibitem{ric}
W.\ Ricker, \emph{Operator algebras generated by commuting projections: a vector measure approach}, Springer, Berlin, 1999.

\bibitem{rud}
W.\ Rudin, \emph{Functional analysis}, $2^{\textup{nd}}$ ed., Tata McGraw-Hill, New Delhi, 1999.

\bibitem{schbook}
H.H.\ Schaefer, \emph{Banach lattices and positive operators}, Springer, New York, 1974.

\bibitem{wil}
D.\ Williams, \emph{Probability with martingales}, Cambridge University Press, Cambridge, 1991.

\bibitem{wnu}
W.\ Wnuk, \emph{Banach lattices with order continuous norms}, Polish Scientific Publishers PWN, Warsaw, 1999.

\bibitem{zaa}
A.C.\ Zaanen, \emph{Riesz spaces. II}, North-Holland, Amsterdam, 1983.


\end{thebibliography}
\end{document}